\documentclass[11pt,a4paper,reqno]{amsart}

\usepackage{amsmath}
\usepackage{amssymb}
\usepackage{amscd}
\usepackage{amsfonts}
\usepackage{mathrsfs}
\usepackage{setspace}
\usepackage{version}
\usepackage{dsfont}
\usepackage{mathtools}
\usepackage{enumitem}

\usepackage{xcolor}

\definecolor{crimson}{rgb}{0.86, 0.08, 0.24}
\definecolor{bleudefrance}{rgb}{0.19, 0.5, 0.91}
\usepackage[pdftex,colorlinks,linkcolor=crimson,citecolor=bleudefrance,filecolor=black]{hyperref}
\usepackage[utf8]{inputenc}

\theoremstyle{plain}
\numberwithin{equation}{section} 
\newtheorem{theorem}{Theorem}[section]
\newtheorem{proposition}[theorem]{Proposition}
\newtheorem{lemma}[theorem]{Lemma}

\newtheorem{corollary}[theorem]{Corollary}

\newtheorem{claim}[theorem]{Claim}
\newtheorem{question}[theorem]{Question}
\theoremstyle{remark}
\newtheorem{remark}[theorem]{Remark}

\usepackage{soul}

\renewcommand{\leq}{\leqslant}
\renewcommand{\geq}{\geqslant}
\newsavebox{\proofbox}
\savebox{\proofbox}{\begin{picture}(7,7)  \put(0,0){\framebox(7,7){}}\end{picture}}

\newcommand\E{\mathbb{E}}
\newcommand\Z{\mathbb{Z}}
\newcommand\R{\mathbb{R}}

\newcommand\p{\mathbb{P}}

\newcommand\N{\mathbb{N}}

\newcommand\inte{\operatorname{int}}

\newcommand\Isom{\operatorname{Isom}}
\newcommand\Lip{\operatorname{Lip}}

\newcommand\supp{\operatorname{supp}}

\renewcommand\P{\mathbb{P}}

\allowdisplaybreaks

\newcommand{\efface}[1]{}

\usepackage{datetime}

\begin{document}

\title[Second order expansion of the rate function at the drift]{Random walks on  hyperbolic spaces: second order expansion of the rate function at the drift}

\author{Richard Aoun}
\address{American University of Beirut, Department of Mathematics, Faculty of Arts and Sciences, P.O. Box 11-0236 Riad El Solh, Beirut 1107 2020, Lebanon (on leave in New York University Abu Dhabi, PO Box 129188, Saadiyat Island, Abu Dhabi, United Arab Emirates)}
\email{ra279@aub.edu.lb}
\thanks{The first author is  supported  by a Research Group Linkage Programme from the Humboldt Foundation}
\author{Pierre Mathieu}
\address{Aix-Marseille Universit\'{e}, CNRS, Centrale Marseille, I2M, UMR 7373, 13453 Marseille, France}
\email{pierre.mathieu@univ-amu.fr}
\author{Cagri Sert}
\address{Institut f\"{u}r Mathematik, Universit\"{a}t Z\"{u}rich, 190, Winterthurerstrasse, 8057 Z\"{u}rich, Switzerland}
\email{cagri.sert@math.uzh.ch}
\thanks{The third author is supported by SNF Ambizione grant 193481.}

\begin{abstract}
Let $(X,d)$ be a geodesic Gromov-hyperbolic space, $o \in X$ a basepoint and $\mu$ a countably supported non-elementary probability measure on $\Isom(X)$. Denote by $z_n$ the random walk on $X$ driven by the probability measure $\mu$. Supposing that $\mu$ has finite exponential moment, we give a second-order Taylor expansion of the large deviation rate function  of the sequence $\frac{1}{n}d(z_n,o)$ and show that the corresponding coefficient is expressed by the variance in the central limit theorem satisfied by the sequence $d(z_n,o)$. This provides a positive answer to a question raised in \cite{BMSS}. The proof relies on the study of the Laplace transform of $d(z_n,o)$ at the origin using a martingale decomposition first introduced by Benoist--Quint together with an exponential submartingale transform and large deviation estimates for the quadratic variation process of certain martingales.
\end{abstract}
\maketitle

%\tableofcontents

\section{Introduction}
Let $(X,d)$ be a geodesic Gromov-hyperbolic space,
% "need" $X$ separable for all except LLN. 
$G=\Isom(X)$, $o$ a base point of $X$. % and for $g \in G$, $\kappa(g):=d(g\cdot o, o)$ the displacement function. 
A probability measure $\mu$ on $G$ defines a random walk on the group $G$ and subsequently on the metric space $X$ in the following way.  Let $(X_i)_{i\in \N}$ be a sequence of i.i.d.~ random variables on $G$ with distribution $\mu$. We let $L_n=X_n \cdots X_1$ denote the successive positions of the random walk on $G$. The  process $(z_n)_{n \in \N}$ on $X$ defined by $z_n=L_n \cdot o$ constitutes a Markov chain on $X$ that we shall refer to as a random walk on $X$. To avoid measurability issues, we will always suppose that the probability measure $\mu$ is countably supported.

Thanks to the subadditive ergodic theorem, under a finite first moment assumption, we have the following law of large numbers
\begin{equation}\label{eq.lln}
\frac{1}{n} d(z_n,o) \underset{n \to \infty}{\overset{a.s.}{\longrightarrow}} \ell_\mu,
\end{equation}
where $\ell_\mu \in [0,\infty)$ is a constant called the drift of the random walk. There has recently been substantial interest in the finer study of asymptotic properties of a random walks on Gromov-hyperbolic spaces. This recent progress shows that the resemblance between the asymptotic behaviour of random walk displacement and classical sums of i.i.d.~ real random variables is far more than the law of large numbers \eqref{eq.lln}: a central limit theorem (CLT) with the optimal finite second moment assumption is proven by Benoist--Quint \cite{BQ.hyperbolic} (see also Horbez \cite{horbez}) improving previous more restrictive versions by Ledrappier \cite{ledrappier.lecturenotes} and Bj\"{o}rklund \cite{bjorklund} -- an alternative proof of the CLT was later given by Mathieu--Sisto \cite{mathieu-sisto} and in a more restrictive setting by Gou\"{e}zel \cite{gouezel.gap}. These show that for a non-elementary probability measure $\mu$ with finite second moment (see below for the definitions), we have
\begin{equation}\label{eq.clt}
\frac{1}{\sqrt{n}}(d(z_n,o) - n l) \underset{n \to \infty}{\overset{\text{law}}{\longrightarrow}} \mathcal{N}(0,\sigma_{\mu}^2).
\end{equation}
% It follows from Horbez Theorem 1.3 and Remark 2.10 that the CLT for $\sigma(g,x)$ holds for at least one $x\in \partially ^{\infty} X$, because $\vu(\partial^{\infty} X)=1$. Then by Lemma 2.6 of Horbez, the CLT for the distance follows.  
The analogue of Cram\'{e}r's theorem on large deviation principles was recently proven by Boulanger--Mathieu--Sert--Sisto \cite{BMSS} (see also Gou\"{e}zel \cite{gouezel.first.moment}): they showed that for a non-elementary probability measure with a finite exponential moment, the sequence $\frac{1}{n}d(z_n,o)$ satisfies a large deviation principle (LDP) with a proper convex rate function $I:[0,\infty) \to [0,\infty]$ vanishing only at the drift $\ell_\mu$: for every (measurable) subset $R$ of $[0,\infty)$, we have 
\begin{equation}\label{eq.ldp}
\underset{\alpha \in \inte(R)}{-\inf I(\alpha)} \leq \underset{n \rightarrow \infty}{\liminf} \frac{1}{n}\ln \mathbb{P}(\frac{1}{n}d(z_n,o) \in R) \leq \underset{n \rightarrow \infty}{\limsup} \frac{1}{n}\ln \mathbb{P}(\frac{1}{n}d(z_n,o) \in R) \leq \underset{\alpha \in \overline{R}}{-\inf I(\alpha)}
\end{equation}
where $\inte(R)$ denotes the interior and $\overline{R}$ the closure of $R$.

Furthermore, concentration inequalities reminiscent of Hoeffding inequalities were recently shown by Aoun--Sert \cite{aoun-sert} and a local limit theorem for random walks on Gromov-hyperbolic group was proven by Gou\"{e}zel \cite{gouezel.local}. 

However, establishing these results analogous to the classical setting of sums of i.i.d.~ real random variables involves overcoming serious issues by use of various approaches and techniques. Apart from mostly geometric approaches such as the ones used in \cite{BMSS,gouezel.first.moment,mathieu-sisto}, two classical methods are present --- say in aforementioned different proofs of the CLT. These are Nagaev's analytic method \cite{nagaev} and Gordin--Lif\v{s}ic's martingale method \cite{martingale.method}. 

Nagaev's method can be seen as a version of the classical Fourier--Laplace transform and it relies on techniques of analytic perturbation theory, and in general, yields sharper estimates. However, implementing it requires proving a certain spectral gap result for a Markov operator acting on an appropriate boundary space. Although this is by-now standard, say, on classical hyperbolic spaces or on free groups, it is not well-developed in the generality of spaces, namely (not necessarily proper) geodesic Gromov-hyperbolic spaces  that we shall we working with in this article. A thorough study of the analytical method in the case of Gromov-hyperbolic groups was done by Gou\"{e}zel \cite[Proposition 3.6, Section 5]{gouezel.gap}.

We will extensively use the martingale approach --- developed in this setting by Benoist--Quint \cite{BQ.CLT, BQ.hyperbolic} and adapted to greater generality by Horbez \cite{horbez} and Aoun--Sert \cite{aoun-sert} --- to tackle the analytic problem of giving a second-order expansion of the limit Laplace transform of the sequence $d(z_n,o)$ (or by convex duality, of its large deviation rate function in \eqref{eq.ldp}) and relating it to the variance in the central limit theorem \eqref{eq.clt}. Similar results are known to hold in settings where spectral methods are available. We now expound on these notions and precisely state the main result of this note. 

A geodesic metric space $(X,d)$ is said to be Gromov-hyperbolic if there exists $\delta>0$ such that  for every $x,y,z,o \in X$, we have $(x|y)_o \geq (x|z)_o \wedge (z|y)_o -\delta$, where $(.|.)_.$ is the Gromov product given by $(x|y)_o=\frac{1}{2}(d(x,o)+d(y,o)-d(x,y))$. A probability measure $\mu$ on $\Isom(X)$ is called \textit{non-elementary} if its support $S$ generates a semigroup that contains two independent loxodromic elements (see \S \ref{sec.large.deviation}). For such a measure $\mu$ and $n \in \N$, we denote by $\mu^{\ast n}$ its $n^{th}$ convolution which is the law of the random variable $L_n$.

Given a probability measure $\mu$, the limit Laplace transform of the sequence $\frac{1}{n}d(z_n,o)$ is the function $\Lambda:\R \to (-\infty,\infty]$ defined by
\begin{equation}\label{eq.defn.laplace}
\Lambda(\lambda)=\lim_{n \to \infty} \frac{1}{n} \log \mathbb{E}[e^{\lambda d(z_n,o)}].
\end{equation}
Note that, since the increments are i.i.d and $G$ acts by isometries on $X$, subadditivity implies that the limit in (\ref{eq.defn.laplace}) exists. Under a finite super-exponential moment assumption, the Fenchel-Legendre transform of $\Lambda$ is the rate function $I$ of the large deviation principle satisfied by $\frac{1}{n}d(z_n, o)$ (see \cite[Lemma C.2]{BMSS}). 
%Laplace transform encodes many properties related to the asymptotic behaviour of random walk displacement sequence $d(z_n,o)$. Notice that when $d(z_n,o)$ is a sum of i.i.d.~ random variables, the quantity inside the limit in the right-hand-side of \eqref{eq.defn.laplace} is constant in ($n \in \N$) and its extension to $\C$ is the usual Fourier-Laplace transform which is the standard tool to study limit theorems in that classical setting. Moreover, when spectral method is available, $\Lambda(\lambda)$ coincides, at least locally, with the logarithm of the spectral radius of the averaging operator. 
%On the hand
%Mention Gartner--Ellis theorem. We also mention that it is the log of the spectral radius of a suitable operator used in the analytic approach. Say Regularity properties not clear. 
Using this and \cite[Theorem 1.1]{BMSS}, one can deduce that the derivative of $\Lambda$ at 0 is equal to the drift $\ell_{\mu}$.

The goal of this note is to prove the following result which answers part of \cite[Question C.1]{BMSS} and which says that the convex function $\Lambda$ has a second order Taylor expansion at $0$ with second derivative equal to the variance in  the central limit theorem: 
\begin{theorem}\label{thm.main}
Let $(X,d)$ be a geodesic Gromov-hyperbolic space and $\mu$ a non-elementary probability measure on $\Isom(X)$. Suppose that $\mu$ has a finite exponential moment, i.e. for some $\alpha>0$, $\int{e^{\alpha \,d(g\cdot o, o)} d\mu(g)}<+\infty$. Then, we have
$$
\lim_{\lambda\to 0} \frac{\Lambda(\lambda)-\lambda \ell_\mu}{\lambda^2}=\frac{\sigma_\mu^2}{2}.$$
\end{theorem}

The proof uses extensively the martingale approach developed in this context by Benoist--Quint \cite{BQ.CLT,BQ.hyperbolic}. The martingale decomposition proved in these works allows us to reduce the study of $\Lambda$ near zero to the study of the limit Laplace transform of a martingale induced by an iid random walk on the group $\Isom(X)$. Once this reduction is done, the proof is divided into two parts: proving the lower bound, i.e.~$\lim_{\lambda\to 0} \frac{\Lambda(\lambda)-\lambda \ell_\mu}{\lambda^2}\geq \frac{\sigma_\mu^2}{2}$ and the upper bound, i.e.~$\lim_{\lambda\to 0} \frac{\Lambda(\lambda)-\lambda \ell_\mu}{\lambda^2}\leq \frac{\sigma_\mu^2}{2}$. 
The proof of the lower bound is based on a new exponential submartingale transform that we establish in Proposition \ref{prop.submart.trans}. The latter extends a classical result of Freedman \cite{Freedman} to the case of martingales with unbounded differences. The proof of the upper bound
uses ideas from martingale concentration inequalities. Another important tool is large deviation estimates for the quadratic variation of our martingales.

\begin{remark}
1. (\textit{Busemann cocycle}) A general version of Theorem \ref{thm.main} will be proved in Theorem \ref{thm.main.tech} where the displacement $d(z_n, o)$ is replaced with the Busemann cocycle $\sigma(L_n, x)$ of $L_n$ based at any point of $x$ in the horofunction compactification of $X$. See also Question  \ref{question} for an ensuing problem.\\[3pt]
2. (\textit{Translation distance}) Thanks to \cite[Theorem 1.3]{BMSS}, when $\mu$ has bounded support, one can replace $d(z_n,o)$ by $\tau(L_n)$ in \eqref{eq.defn.laplace}, where $\tau(.)$ denotes the translation distance given for $g \in \Isom(X)$ by $\tau(g)=\lim_{n \to \infty}\frac{1}{n}d(g^n \cdot o,o)$.\\[3pt]
3. (\textit{Positivity of $\sigma_\mu$}) By an argument of Benoist--Quint \cite{BQ.hyperbolic}, it follows from the expression of $\sigma_\mu^2$ (see \eqref{def.variance}) that $\sigma_\mu>0$ if any only if $\mu$ is non-arithmetic (see Remark \ref{rk.sigma.positive}).
\end{remark}

Using the convexity of the rate function proved in \cite{BMSS} and standard results from convex analysis, we deduce
\begin{corollary}[About the rate function]\label{corol.rate.function}
Keep the assumptions of Theorem \ref{thm.main} and let $I$ be the rate function \eqref{eq.ldp}. Then, we have
$$
\lim_{\lambda \to 0}\frac{I(\ell_\mu+\lambda)}{\lambda^2}=\frac{1}{2\sigma_\mu^2}$$
%the rate function $I$ has a second order expansion    at $\ell_{\mu}$ with $I''(\ell_{\mu})=\frac{1}{\sigma_{\mu}^2}$, i.e.   $I(\ell_\mu+\lambda)\underset{\lambda \to 0}{=
%}\frac{1}{2\sigma_\mu^2}\lambda^2+ o(\lambda^2)$
\end{corollary}

Finally, we note that our results are also valid for the right random walk $R_n=X_1\ldots X_n$ since for every $n \in \N$, $L_n$ and $R_n$ have the same distribution.

The paper is organized as follows. In Section \ref{sec.martingale}, we recall some preliminaries on submartingales and prove an exponential transform for submartingales. In Section \ref{sec.hyperbolic}, we recall basic definitions about Gromov-hyperbolic spaces and metric compactifications as well as results from the theory of random walks on hyperbolic spaces. In particular, we recall that $d(z_n, o) -n\ell_{\mu}$ is at bounded distance from a martingale and prove a large deviation estimate for the predictable quadratic variation of the latter. In Section \ref{sec.proof} we prove Theorem \ref{thm.main} in its general form Theorem \ref{thm.main.tech},  by treating separately the lower bound (\S \ref{subsec.lower.bound}) and the upper bound (\S \ref{subsec.upper.bound}). In \S \ref{subsec.corol}, we deduce Corollary \ref{corol.rate.function}, and finally, discuss some ensuing questions in \S \ref{subsec.questions}. 
%The lower bound in Theorem \ref{thm.main} will be proved in   Section
% \ref{subsec.upper.bound} while the upper bound will be proved in \ref{subsec.lower.bound}. In Section \ref{sec.conclusion} we deduce Theorem \ref{thm.main} (in a more general form) and Corollary \ref{corol.rate.function}, and give ensuing questions. 
%we prove the upper bound in Theorem \ref{thm.main}we prove our main results Theorem \ref{thm.main} and Corollary \ref{corol.rate.function}.}

\section{Preliminaries on martingales}\label{sec.martingale}

In this section, we recall some preliminaries from the theory of martingales and prove a result about exponential martingale transforms that will play a crucial role in the proof of our main theorem.

%These results are extensions of those of Freedman in his seminal work \cite{Freedman} to the case of unbounded martingale differences. Indeed, Freedman's results would be sufficient for our purposes under the assumption that the probability measure $\mu$ driving the random walk is boundedly supported. However, these extensions are needed to handle the case of finite exponential moment. 

%We will need two ingredients corresponding to the extensions of Freedman's \cite[Corollary 1.4]{Freedman}. The relevant submartingale transform (the main ingredient of Bernstein--Bennett type inequalities for martingales) is due to Dzhaparidze--van Zanten \cite{DZ} (see also the work of Fan--Grama--Liu \cite{fan.grama.liu}) which we now recall.

Let us first fix our notation. We shall denote by $\mathcal{F}=(\mathcal{F}_n)_{n\in \N}$ an increasing sequence of $\sigma$-algebras (a filtration) on a fixed standard  probability space $\Omega$. Usually, we will consider the filtration to be fixed and omit it from the notation. The notation $M=(M_n)_{n \in \N}$ will be reserved for an adapted sequence of random variables that form either a martingale or submartingale. Denoting by $\Delta M$ the sequence of differences given by $\Delta_n M:=M_n-M_{n-1}$, we recall that $M$ being a submartingale means that for every $n\in \N$, $M_n$ is $\mathcal{F}_n$-measurable, integrable, and it satisfies respectively $\mathbb{E}[\Delta_n M | \mathcal{F}_{n-1}] \geq 0$. In the sequel, unless otherwise stated, we take $M_0=0$ a.s. The predictable quadratic variation (or conditional quadratic variation) of the submartingale $M_n$ is denoted by $\langle M\rangle_n:= \sum_{i=1}^n \mathbb{E}[(\Delta_i M)^2|\mathcal{F}_{i-1}]$. Given a positive constant $a>0$, we denote 
\begin{equation}\label{transform.G}
G_n^a=\sum_{i=1}^n \mathbb{E}[(\Delta_i M)^2 1_{|\Delta_i M| \leq a}| \mathcal{F}_{i-1}]-a \sum_{i=1}^n |\Delta_i M| 1_{|\Delta_i M| \geq a}.\end{equation} Finally, the following special function defined on $\R$ will play a significant role: $\mathfrak{f}(\lambda)=e^{-\lambda} -1 +\lambda$.

We start by recalling Freedman's submartingale transform whose statement and proof strategy will be used in our generalization below.

\begin{proposition}\cite[Corollary 1.4 (b) \& (3.9)]{Freedman} \label{prop.freedman.submartingale}
1. Let $X$ be an integrable random variable with $\mathbb{E}(X)=0$ (resp.~ $\mathbb{E}[X] \geq 0$) and $X\geq -1$ (resp.~ $|X| \leq 1$) a.s. Then, for every $\lambda \geq 0$, we have $$
\mathbb{E}[e^{\lambda X}] \geq e^{\mathfrak{f}(\lambda) Var(X)}.
$$
2. Let $(M_n)_{n\in \N}$ be a submartingale and such that for every $1 \leq n \in \N$, $|\Delta_n M|\leq 1$ almost surely. Then for every $\lambda \geq 0$, the sequence of random variables 
$$\left(\exp\left(\lambda M_n -  \mathfrak{f}(\lambda) \langle M\rangle_n\right)\right)_{n \in \N}$$ is a submartingale with respect to the same filtration.
\end{proposition}

We note that the second statement above is a consequence of the first one.

%We will use the  the previous theorem above (and its proof's strategy), to deduce the following result for submartingales with not necessarily unbounded differences. 

The following result provides a generalization of Proposition \ref{prop.freedman.submartingale} to submartingales with increments possessing a finite exponential moment.

\begin{proposition}[Submartingale transform]\label{prop.submart.trans}
Let $(M_n)_{n\in \N}$ be a submartingale. Suppose that there exists a constant $\alpha>0$ such that for every $n \in \N$, we have $\mathbb{E}[e^{\alpha \sum_{k=1}^n |\Delta_k M|}]<\infty$. Then, given any $a>0$, for every $\lambda >0$ small enough, the sequence of random variables $$\left(\exp\left(\lambda M_n - \frac{\mathfrak{f}(\lambda a)}{a^2}G_n^a\right)\right)_{n \in \N}$$ is a submartingale with respect to the same filtration. 
\end{proposition}

This is an extension (to unbounded differences) of Freedman's submartingale transform in his seminal work \cite{Freedman}. Indeed, if the difference sequence $\Delta_n M$ satisfies $|\Delta_n M| \leq 1$ a.s., the transform in the previous result boils down to Proposition \ref{prop.freedman.submartingale}. On the other hand, it applies, for instance, when there exists a constant $\alpha>0$ such that for every $n\in \N$,  $\mathbb{E}[e^{\alpha |\Delta_n M|} | \mathcal{F}_{n-1}]<+\infty$. 
 This will be the case in our application.
The counterparts of Proposition \ref{prop.submart.trans} for supermartingale transforms were obtained by Dzhaparidze--van Zanten \cite{DZ} (see also Fan--Grama--Liu \cite{fan.grama.liu}).

\begin{proof} Let $a>0$. By the finite exponential moment hypothesis on $(M_n)_{n \in \N}$, it is clear that for every $\lambda>0$ small enough and for every $n \in \N$,  $\exp(\lambda M_n - \frac{\mathfrak{f}(\lambda a)}{a^2}G_n^a)$ is $\mathcal{F}_n$-measurable and integrable. Therefore, by expanding the conditional expectation, one sees that it is enough to show the following: for any integrable random variable $X$ with $\E[X]\geq 0$, for any $\lambda>0$, 
\begin{equation}\label{inequality.super}
\E[\exp(\lambda X + \frac{\mathfrak{f}(\lambda a)}{ a} |X| \mathds{1}_{|X|\geq a})]
\geq e^{\frac{\mathfrak{f}(\lambda a)}{a^2} \E[X^2 \mathds{1}_{|X| \leq a}]}.\end{equation}
Denote by $\nu$ the distribution of $X$. 
\begin{itemize}[leftmargin=1cm]
\item Case 1: $\E[X]=0$ and $\nu$ is supported on two points $-c$ and $d$ with $c,d>0$ and both $c,d \leq a$. Let $r:=\max\{c,d\}$. 
Since $\E[X]= 0$ and 
$X\geq -r$ almost surely,  1.~ of   Proposition \ref{prop.freedman.submartingale} (applied to  $\frac{X}{r}$ and to $\lambda r$) yields   
$$\E[\exp(\lambda X + \frac{\mathfrak{f}(\lambda a)}{ a} |X| \mathds{1}_{|X|\geq a})]= \E[\exp(\lambda X)]\geq e^{\frac{\mathfrak{f}(\lambda r )}{r^2} \E[X^2]}.$$
Since the function $x\mapsto \frac{\mathfrak{f}(x)}{x^2}$ is decreasing on $\R$ and since $r \leq a$, we deduce that $\frac{\mathfrak{f}(\lambda r)}{r^2}\geq \frac{\mathfrak{f}(\lambda a)}{a^2}$. Therefore
$$
\E[\exp(\lambda X + \frac{\mathfrak{f}(\lambda a)}{a} |X| \mathds{1}_{|X|\geq a})]\geq e^{\frac{\mathfrak{f}(\lambda a) }{a^2}  \E[X^2]} = e^{\frac{\mathfrak{f}(\lambda a )}{a^2} \E[X^2 \mathds{1}_{|X|\leq a}]}.
$$

\item Case 2: $\E[X]=0$ and $\nu$ is supported exactly on two points $-c$ and $d$ and we are not in Case 1. By Jensen's inequality, we obtain 
\begin{equation}\label{eq.i1}
\begin{aligned}
\E[\exp(\lambda X +\frac{\mathfrak{f}(\lambda a)}{ a} |X| \mathds{1}_{|X|\geq a})] &\geq \exp(\lambda \E[X]+ \frac{\mathfrak{f}(\lambda a)}{a} \E[|X|\mathds{1}_{|X|\geq a}])\\ &=
\exp( \frac{\mathfrak{f}(\lambda a)}{a} \E[|X| \mathds{1}_{|X|\geq a}]).
\end{aligned}
\end{equation}
If both $c,d\geq a$, then the right hand side of \eqref{inequality.super} is equal to $1$ and hence \eqref{inequality.super} holds in view of \eqref{eq.i1}. So, since we are also not in Case 1, we can suppose that either $c > a$ and $d< a$, or $d > a$ and $c< a$. Let us treat the case $c >  a$ and $d < a$. Notice also that since $\E[X]=0$ we have $\p(X=c)=\frac{d}{c+d}$ and $\p(X=d)=\frac{c}{c+d}$.
By assumption on $c,d$ and \eqref{eq.i1}, these yield 
\begin{equation}\label{inequality.super1}\E[\exp(\lambda X +\frac{\mathfrak{f}(\lambda a)}{ a} |X| \mathds{1}_{|X|\geq a})]
\geq \exp(\frac{\mathfrak{f}(\lambda a)}{a} \frac{cd}{c+d}).\end{equation}
On the other hand, 
\begin{equation}\label{inequality.super2}\exp \left( \frac{\mathfrak{f}(\lambda a)}{a^2} \E[X^2 1_{|X|\leq a}]\right)=
\exp\left(\frac{\mathfrak{f}(\lambda a)}{a^2}   \frac{d^2c}{c+d}\right).\end{equation}
Since $d\leq a$, \eqref{inequality.super} follows from combining \eqref{inequality.super1} and \eqref{inequality.super2}. The case $d>a$ and $c\leq a$ can be treated similarly. 
   
\item Case 3: $\E[X]\geq 0$ and $\nu$ is supported on two points $\{-c,d\}$ with $c,d>0$. We will study the behavior of the left-hand-side and the right-hand-side of \eqref{inequality.super} when we vary $\nu$ with the condition $\E[X]\geq 0$, while fixing $a,c,d,\lambda$. Since $\nu$ is supported on two points, it is enough to treat the behavior of these quantities when 
$\beta:=\p(X=d)$ varies. Observe that since $\E[X]\geq 0$, we have $ \frac{c}{c+d}\leq \beta \leq 1$. The function  
$\psi_1(\beta):=\E_{\nu}[\exp(\lambda X + \frac{\mathfrak{f}(\lambda a)}{a} |X| \mathds{1}_{|X|\geq a})]$ is linear in $\beta$ while the function $\psi_2(\beta):= e^{\frac{\mathfrak{f}(\lambda a)}{a^2} \E_{\nu}[X^2 \mathds{1}_{|X| \leq a}]}$ is convex in $\beta$ (being of the form $\psi_2(\beta)=e^{L(\beta)}$ with $L$ an affine map). Thus the function $\psi:=\psi_1-\psi_2$ is concave on $[\frac{c}{c+d}, 1]$. It suffices then to check that $\psi(\frac{c}{c+d})\geq 0$ and that $\psi(1)\geq 0$. The case $\beta=\frac{c}{c+d}$ corresponds to the case $\E_{\nu}[X]=0$ and hence, by cases 1 and 2, $\psi(\frac{c}{c+d})\geq 0$. 
The case $\beta=1$ corresponds to $\nu=\delta_d$. Clearly, $\psi(1)\geq 0$ when  $d\geq a$.  When $d<a$, the relation $\psi(1)\geq 0$    follows from the facts that the function  $x\mapsto \frac{\mathfrak{f}(x)}{x^2}$ is decreasing and that $\mathfrak{f}(x)\leq x$ for every $x\geq 0$. This concludes the proof of  \eqref{inequality.super} in this case. 

\item Case 4: here we treat the general case (cf.~ proof of \cite[Proposition 3.6]{Freedman}). Since $\E[X]\geq 0$, we can find a family $(\nu_{\alpha})_{\alpha \in I}$ of probability measures, each supported on two points $-c_{\alpha}\leq 0$ and $d_{\alpha}>0$ and  of expectation $\geq 0$, and a probability measure $\theta$ on $I$ such that $\nu=\int{d\theta(\alpha) \nu_{\alpha}}$. We have
\begin{equation*}
\begin{aligned}
\E[\exp(\lambda X + \frac{\mathfrak{f}(\lambda a)}{a} |X| \mathds{1}_{|X|\geq a})]&=\int
\left(\int  e^{\lambda x +\frac{\mathfrak{f}(\lambda a)}{a}|x| 1_{|x| \geq a}}d\nu_\alpha(x)\right)d\theta(\alpha) \\  
&\geq \int  e^{\frac{\mathfrak{f}(\lambda a)}{a^2} \int x^2 1_{|x| \leq a} d\nu_\alpha(x)}d\theta(\alpha) \\ 
& \geq e^{\frac{f(\lambda a)}{a^2} \iint{x^2 \mathds{1}_{|x|\leq a} d\nu_{\alpha}(x) d\theta(\alpha)}}\\
&=   e^{\frac{\mathfrak{f}(\lambda a)}{a^2} \mathbb{E}[X^2 1_{|X| \leq a}]},
\end{aligned}
\end{equation*}
where we applied \eqref{inequality.super} for each probability measure $\nu_{\alpha}$ in the second inequality and Jensen in the third inequality. 
\end{itemize}
\end{proof}

\section{Random walks on hyperbolic spaces}\label{sec.hyperbolic}

\subsection{Preliminaries on hyperbolic spaces}\label{subsec.hyperbolic}
Let us first fix our notation. 
Let $(X,d)$ be a geodesic metric space. Fix a base point $o\in X$.  Recall that $(X,d)$ is said to be $\delta$-hyperbolic (where $\delta\geq 0$) if 
  for every $x,y,z,o \in M$,
\begin{equation}\label{eq.defining.eq}
(x|y)_o \geq (x|z)_o \wedge (z|y)_o -\delta,    
\end{equation}
where $(.|.)_.$ is the Gromov product given by $(x|y)_o=\frac{1}{2}(d(x,o)+d(y,o)-d(x,y))$. For simplicity, we will often omit the basepoint $o$ from the notation. 
We recall that this category of metric spaces comprises many usual spaces: trees, classical hyperbolic spaces, the fundamental group of  compact surfaces of genus $\geq 2$. We recall that the definition of hyperbolicity is equivalent to geodesic triangles  being  thin. We refer to \cite{hyperbolic-book} for general properties of these spaces. Denote by $G:=\Isom(X)$ the group of isometries of the metric space $(X,d)$. The displacement of $g\in G$ is by definition 
$$\kappa(g):=d(g\cdot o, o).$$
An element $\gamma \in G$ is said to be loxodromic if for any $x \in X$, the sequence $(\gamma^nx)_{n \in \Z}$ constitutes a quasi-geodesic (see \cite[Ch.\ 3]{hyperbolic-book}). Equivalently, $\gamma$ is loxodromic if and only if it fixes precisely two points $x_\gamma^+,x_\gamma^-$ on the Gromov boundary $\partial X$ of $X$ \cite[Ch.\ 9 \& 10]{hyperbolic-book}. Two loxodromic elements $\gamma_1,\gamma_2$ are said to be independent if the sets of fixed points $\{x^+_{\gamma_i} ,x^-_{\gamma_i}\}$ for $i=1,2$ are disjoint. Finally, a set $S$, or equivalently a probability measure with support $S$, is said to be non-elementary if the semigroup generated by $S$ contains at least two independent loxodromic elements.

Now we recall the definition of the Busemann compactification of $X$ (no need for hyperbolicity in this part). Denote by  $\Lip^1(X)$ the set of real valued Lipschitz functions on $X$ with Lipschitz constant $1$, endowed with the  topology of pointwise convergence. Fixing $o \in X$, for $x \in X$, let the function $h_x \in \Lip_{o}^1(X)$, defined by $h_x(m)=d(x,m)-d(x,o)$, where $\Lip_{o}^1(X)$ is the subspace of $\Lip^1(X)$ consisting of functions $f$ satisfying $f(o)=0$. If $X$ is separable, the closure of $\{h_x  \, ; \, x \in X\}$ is a compact metrizable subset of $\Lip^1_o(X)$, called the \textit{horofunction compactification} of $X$ (see e.g.\ \cite[Proposition 3.1]{maher-tiozzo}). It will be denoted as $\overline{X}^h$. The map $x \mapsto h_x$ is injective on $X$ (and an embedding when $X$ is a proper metric space) and we usually identify $X$ with its image in $\overline{X}^h$. The \textit{horofunction boundary} of $X$ is defined as $\partial_h X:=\overline{X}^h\setminus X$. The group of isometries $\Isom(X)$ acts on $\overline{X}^h$ by homeomorphisms given, for $g \in \Isom(X)$,  $h \in \overline{X}^h$ and $m \in X$, by $(g.h)(m)=h(g^{-1}m)-h(g^{-1}o)$. This extends equivariantly the isometric action of $\Isom(X)$ on $X$ and the set $\partial_h X \subset \overline{X}^h$ is invariant under $\Isom(M)$. The \textit{Busemann cocycle} $\sigma: \Isom(X) \times \overline{X}^h \to \mathbb{R}$ is defined by $$\sigma(g,h)=h(g^{-1}o).$$
Observe that for every $g\in G$ and $x\in \overline{X}^h$, 

\begin{equation}\label{useful}\sigma(g,o)=\kappa(g) \,\,\,\,\textrm{and}\,\,\,\, |\sigma(g,x)|\leq \kappa(g)\end{equation}

Finally, we recall that the Gromov product can be extended to the whole Busemann compactification by setting $(x|y)_o:=-\min_{z\in X}{\frac{1}{2}(h_x(z)+h_y(z))}$. In particular, one can infer that for $x\in \overline{X}^h$ and 
$y\in X$,   \begin{equation}\label{gromov.extended.space.boundary}(x|y)_o=\frac{1}{2} (d(y,o)-h_x(y)).\end{equation}

\subsection{Random walks}
\label{sec.large.deviation}
There are two main goals in this section. The first one (discussed in \S \ref{subsub.mart.dec}) is to recall a martingale decomposition (Lemma \ref{lemma.martingale.dec}) of the Busemann cocycle along non-elementary random walks on Gromov-hyperbolic spaces which is due to Benoist--Quint \cite{BQ.CLT,BQ.hyperbolic} (see also an extension in \cite{horbez}). We will use a slightly more general version of this worked out in \cite{aoun-sert}. The second goal (discussed in \S \ref{subsub.bracket.large.dev}) is to prove Proposition \ref{prop.bracket.large.dev} about large deviations of predictable quadratic variation and its consequence expressed in Corollary \ref{corollary.laplace.bracket}. The latter will be crucial in the proof of our main result.

\subsubsection{Benoist--Quint martingale decomposition}\label{subsub.mart.dec} 

Let $\mu$ be a probability measure on the isometry group $G$ of $X$ with countable support. Recall that it is said to have a finite exponential moment (resp.~ finite second moment) if there exists $\alpha>0$ such that $\int e^{\alpha d(g \cdot o, o)} d\mu(g)<\infty$ (resp.~ $\int \kappa(g)^2 d\mu(g)<\infty$). Let $L_n=X_n\cdots X_1$ be the random walk on $G$ and $\ell_{\mu}$ the drift of the random walk on $X$ defined in \eqref{eq.lln}.
Denote by $\mathcal{F}$ the natural filtration generated by the increments $X_i$'s. Finally, we denote by $P_{\mu}$ the Markov operator on the horofunction compactification $\overline{X}^h$  induced by the random walk on $G$, i.e.~ $P_{\mu}f(x)=\int{f(g\cdot x) d\mu(g)}$ for every bounded measurable function $f$ on $\overline{X}^h$. The starting point of the proof of Theorem \ref{thm.main} is the following.

\begin{lemma}\label{lemma.martingale.dec}
Let $\mu$ be a non-elementary probability measure  with finite second moment. Then, for every $x\in \overline{X}^h$, there exists a martingale $M_x=(M_{x,n})_{n\in \N}$ with respect to the filtration $\mathcal{F}$ starting at the origin and such that  for every $n\in \N$, 
$$\sigma(L_n, x)-n\ell_{\mu} = M_{x,n} + O_{x,n}(1),$$
where $O_{x,n}(1)$ is a random variable whose absolute value is bounded uniformly in $n\in \N$ and $x\in \overline{X}^h$. 
\end{lemma}

\begin{proof}
When $X$ is proper,  Benoist--Quint   \cite[Proposition 4.6]{BQ.book} showed that that there exists a bounded measurable function $\psi$ on the Busemann boundary $\partial^h X$ such that $$\psi - P_{\mu} \psi = \int{\sigma(g,x) d\mu(g)} - \ell_\mu.$$ It was then verified in \cite{horbez} that this solution can be extended to the case when $X$ is non-proper and also in \cite{aoun-sert} that $\psi$ could be defined on the whole compactification $\overline{X}^h$ while preserving the boundedness of $\psi$.  This is equivalent to finding a cocycle $\sigma_0: G\times \overline{X}^h \to \R$ with constant drift equal to $\ell_{\mu}$, i.e.~ $$\int{\sigma_0(g,x) d\mu(g)}=\ell_{\mu}$$ for every $x\in \overline{X}^h$, such that the following identity holds for every $(g,x)\in G\times  \overline{X}^h $:  
\begin{equation}\label{cocycle.decomposition}\sigma(g,x) = \sigma_0(g,x) - \psi(g\cdot x)+\psi(x).\end{equation} 
Let then 
\begin{equation}\label{equation.martingale}M_{x,n}:=\sigma_0(L_n, x)-n\ell_{\mu}.\end{equation}
The constant drift property of $\sigma_0$ implies that $M_x:=(M_{x,n})_{n\in \N}$ is a martingale with respect to the filtration $\mathcal{F}$, which finishes the proof. 
\end{proof}

\begin{remark}\label{remark.martingale}
Observe that since $\E(M_{x,n})=\E(M_{x,0})=0$ for every $n\in \N$, we obtain the existence of some $C\geq 0$ such that for every $n\in \N$ and every $x\in M$, 
$$n\ell_\mu - C \leq \E[\sigma(L_n, x)] \leq n \ell_\mu+C.$$
\end{remark}

From now on, for every $x\in \overline{X}^h$ we denote by $M_x=(M_{x,n})_{n\in \N}$ the martingale defined in the proof of Lemma \ref{lemma.martingale.dec}, i.e. 
 $$M_{x,n}:=\sigma_0(L_n, x)-n\ell_{\mu}.$$
 
Many properties of a martingale are encoded in its different notions of   quadratic variation. For instance, a martingale whose predictable quadratic variation (see below for the definition) is almost surely bounded satisfies a Bennett--Bernstein  concentration result (see \cite{Freedman} for the bounded difference case and \cite{pena,DZ,fan.grama.liu} for the general case). Burkholder inequalities \cite{burkholder} are another instance of the relevance of the quadratic variation in studying martingales. 
 
\subsubsection{Large deviation estimate for predictable quadratic variation of $M_{x,n}$} \label{subsub.bracket.large.dev} 
We now proceed with the second goal of \S \ref{sec.large.deviation}, namely proving Proposition \ref{prop.bracket.large.dev} below and deducing Corollary \ref{corollary.laplace.bracket}. We first give some observations and definitions regarding the martingale $(M_{x,n})_{n\in \N}$ introduced in \S \ref{subsub.mart.dec}. The martingale difference of $(M_{x,n})_{n\in \N}$ is 
\begin{equation}\label{our.martingale.difference}\Delta_n M_x:=M_{x,n}-M_{x,n-1}=\sigma_o(X_n, Z_{x,n-1})-\ell_\mu,\end{equation}
 where   $(Z_{x,j}:=L_j\cdot x)$ is the Markov chain on $\overline{X}^h$  induced by the random walk on $G$ and starting at $x$. 
We recall that the (predictable) quadratic variation of $\langle M_x\rangle$ is the unique increasing predictable process  such that $(M_{x,n}^2-\langle M_x \rangle_n)_{n\in \N}$ is a martingale. We have 

\begin{equation}\label{bracket.expression}\langle M_x\rangle_n=
\sum_{j=1}^n{\E(\Delta_j M_{x}^2 | \mathcal{F}_{j-1})}=
\sum_{j=1}^n{\int{(\sigma_0(g,Z_{x,j-1}) -\ell_{\mu})^2 d\mu(g)}}.\end{equation}

We now come to the main result of this section. Its statement contains the expression  
\begin{equation}\label{eq.exp.variance}
\sigma_{\mu}^2:=\iint{(\sigma_0(g,x)-\ell_{\mu})^2 d\mu(g) d\nu(x)}
\end{equation}
where $\nu$ is any $\mu$-stationary probability measure on $\overline{X}^h$ -- we will see that the integral does not depend on $\nu$. This constant $\sigma_\mu^2$ is also the variance appearing in the central limit theorem \eqref{eq.clt} (see proof of \cite[Theorem 4.7.b]{BQ.hyperbolic} or \cite[Theorem 1.3]{horbez}). 
 
\begin{proposition}[Large deviation estimates for the quadratic variation]
\label{prop.bracket.large.dev} Let $\mu$ be a non-elementary probability measure with finite second moment. Then for every $\epsilon>0$ 
$$\limsup_{n\to \infty}\frac{1}{n} \log \sup_{x\in X}\p(| \langle M_x \rangle_n - n \sigma_{\mu}^2 | \geq n \epsilon)<0.$$
\end{proposition}
% We are using Markov Feller operators since $\Gamma_{\mu}$ is countable.
%In the sequel, $\sigma_{\mu}^2$ will denote the constant defined by the previous result and its expression will be given in Lemma \ref{lemma.unif.conv}. 

To proceed to prove this result, we first observe that we can reformulate the statement as a statement about large deviations for an additive functional of a Markov chain. Indeed, for $x\in \overline{X}^h$ defining 
\begin{equation}\label{eq.def.phi}
\phi(x):=\int{(\sigma_0(g,x) -\ell_{\mu})^2 d\mu(g)}
\end{equation}
expression \eqref{bracket.expression} shows that $$\langle M_x \rangle_n=\sum_{j=1}^n{\phi(Z_{x,j})}.$$
  
Benoist--Quint showed a large deviation estimate for functionals along  Markov Chains \cite[Proposition 3.1]{BQ.CLT}, which is a quantitative refinement of Breiman's law of large numbers. In the aforementioned paper, the authors work with continuous functions in the framework of  Markov--Feller operators on compact metric spaces. However, in the generality that we work with, we were not able to prove the continuity of $\phi$. Note that by the expression \eqref{cocycle.decomposition} of the cocycle $\sigma_0$, the continuity of $\phi$ would follow from the continuity %(even separately in variables)
of the Gromov product on the Busemann compactification $\overline{X}^h$.
Up to our knowledge, the latter is known in familiar cases including trees and classical hyperbolic spaces but not in our generality  (note that by \cite[\S  10]{miyachi} the Gromov product on the Busemann compactification of a general metric space $X$ may fail to be continuous even if $X$ is proper and geodesic). %Nevertheless, this does not eliminate the possibility for it to be continuous if we assume moreover $X$ to be hyperbolic. 
To  overcome this issue, we will adapt the statement of Benoist--Quint by relaxing the continuity assumption.

\begin{proposition}(\cite[Proposition 3.1]{BQ.CLT} modified)\label{proposition.deviations.MC}
Let $(Z_n)_{n\in \N}$ be a Markov chain on a state space $E$, $P$ its Markov operator and $\phi: E\to \R$ a measurable bounded function. Suppose that 
\begin{equation}\label{uniform}\frac{1}{n}\sum_{j=1}^n{P^j\phi(x)}\to l_{\phi}\in \R,\end{equation} uniformly in $x\in E$. Then the following large deviation estimate holds: for every $\epsilon>0$ 
$$\limsup_{n\to +\infty}\frac{1}{n}\log \p_x(
\sum_{i=1}^n\phi(Z_i) \in [nl_{\phi} - n\epsilon , nl_{\phi} +n\epsilon])<0,
$$ 
uniformly in $x\in E$. 
\end{proposition}
The proof is an adaptation of Benoist--Quint's proof of \cite[Proposition 3.1]{BQ.CLT}. We include it for the convenience of the reader.
\begin{proof}
We begin with  a general result, which can be seen as a quantitative version of the law of large numbers stated in \cite[Theorem 1.6]{BQ.book}.  If $(\xi_n)_{n\in \N}$ is a sequence of bounded  real random variables adapted to a filtration $\{\mathcal{F}_n \, | \, n\in \N\}$ then 
\begin{equation}\label{inter}\p(|\sum_{j=1}^n( \xi_j - \E[\xi_j| \mathcal{F}_{j-1}]| \geq n\epsilon)\leq 2 \exp(-n\epsilon^2/8||\xi||_{\infty}^2).\end{equation}
Indeed the sequence $\{\xi_n-\E[\xi_n| \mathcal{F}_{n-1} ] \, |\, n\in \N\}$ is a bounded martingale difference sequence with respect to the filtration $\mathcal{F}$ and hence \eqref{inter} follows from    from Azuma--Heoffding's concentration inequality for martingales with bounded differences.

In the second step, we show that for every $m\in \N$,  $\frac{1}{n}\sum_{j=1}^n{\phi(Z_j)}$ concentrates around the Ces\`{a}ro average  $\frac{1}{m}\sum_{j=1}^m{P^j \phi(Z_j)}$; more precisely for every $\epsilon>0$, $n,m\in \N$,
\begin{equation}\label{toshow}
\p(|\sum_{j=1}^{n}
{[\phi(Z_{j})-\frac{1}{m}\sum_{l=1}^mP^l\phi(Z_{j})}]| \geq m n \epsilon+2m||\phi||_{\infty})\leq 2 m^2 
 \exp(-n\epsilon^2/8||\phi||_{\infty}^2).\end{equation}
Indeed, let $1\leq l \leq m$. We write 
$$\sum_{j=l+1}^{l+n}
{(\phi(Z_{j})-\E[\phi(Z_j |  \mathcal{F}_{j-l})])}=
\sum_{k=0}^{l-1} \sum_{j=l+1}^{l+n}{
(\E[\phi(Z_j)| \mathcal{F}_{j-k}]-
\E[\phi(Z_j)|
\mathcal{F}_{j-k-1}]}),
$$
where $\mathcal{F}$ is the filtration induced by the Markov chain. 
For each $k\in \{0,\cdots, l-1\}$,  we apply \eqref{inter} with 
the sequence 
of random variables 
$\{\xi_{j,k}=\E[\phi(Z_j)| \mathcal{F}_{j-k}] \, | \, j\in \N\}$  which are adapted  to the filtration 
$\{\mathcal{F}_{j-k}|j\in \N\}$ and bounded by $||\phi||_{\infty}$. Combining the resulting $l$ estimates, we obtain that  
$$\p(|\sum_{j=l+1}^{l+n}
(\phi(Z_{j})-\E[\phi(Z_j)|\mathcal{F}_{j-l}]) \geq l n \epsilon)\leq 2 l 
\exp(-n\epsilon^2/8||\phi||_{\infty}^2).$$
Noticing that $\E[\phi(Z_j)|\mathcal{F}_{j-l}]=P^l\phi(Z_{j-l})$, the previous estimate gives (after killing the boundary terms using the boundedness of $\phi$) that  
$$\p(|\sum_{j=1}^{n}
{[\phi(Z_{j})-P^l\phi(Z_{j})}| \geq l n \epsilon+2l||\phi||_{\infty})\leq 2l 
\exp(-n\epsilon^2/8||\phi||_{\infty}^2).$$
Estimate \eqref{toshow} immediately follows.

Finally, we use the uniform convergence \eqref{uniform} in order to conclude.  Indeed, the latter condition yields an integer   $m_0$ such that almost surely for every $j\in\N$, 
\begin{equation}\label{uniform.consequence}\frac{1}{m_0}\sum_{l=1}^{m_0} {P^l\phi(Z_{j})}\in [l_{\phi}-\epsilon, l_{\phi}+\epsilon].\end{equation} 
Plugging \eqref{uniform.consequence} into \eqref{toshow} with this $m_0$ immediately finishes  the proof. \end{proof}
 
\begin{remark}
If $E$ is a compact metric space, $P$ a Markov Feller operator and $\phi$ is a continuous function which has a unique average with respect to stationary measures on $E$, then \eqref{uniform} is fulfilled. As mentioned earlier, this is the case, for instance, for  random walks on trees, classical hyperbolic spaces and also for strongly irreducible and proximal random walks on projective spaces (see for instance \cite{BQ.book}). 
\end{remark}

We now check that \eqref{uniform} is satisfied for our function $\phi$ defined in \eqref{eq.def.phi} and the Markov operator $P=P_\mu$ of the Markov chain on $\overline{X}^h$ induced by the random walk on $G$ (see  \S \ref{sec.large.deviation}). 

\begin{lemma}\label{lemma.unif.conv}
Let $\phi:\overline{X}^h \to \R$ as defined in \eqref{eq.def.phi}. Then the sequence of functions $f_n(x):=\frac{1}{n}\sum_{j=1}^n{P^j_\mu \phi(x)}$ converges uniformly on $\overline{X}^h$ to $\sigma^2_\mu\geq 0$. The limit $\sigma_{\mu}^2$ can be expressed as 
\begin{equation}\label{def.variance}\sigma_{\mu}^2:=\iint{(\sigma_0(g,x)-\ell_{\mu})^2 d\mu(g) d\nu(x)},
\end{equation}
where $\nu$ is any $\mu$-stationary measure on $\overline{X}^h$. 
\end{lemma}

\begin{remark}\label{rk.sigma.positive}
It follows from \eqref{def.variance} and the argument in the proof of \cite[Theorem 4.7.b]{BQ.hyperbolic} that $\sigma_\mu=0$ if any only if there exists a constant $C>0$ such that for every $n \in \N$ and $g\in \supp(\mu^{\ast n})$, we have $|\kappa(g)-n\ell_{\mu}|\leq C$. It follows that $\sigma_\mu>0$ if any only if $\mu$ is non-arithmetic.
%Ingredients: continuity of the Busemann cocycle $\sigma$, Lemma 2.1 of Horbez (choice of basis, replacement of Equation 4.3 of Benois--Quint in the proper case) and the fact that $\nu(\partial^{\infty} X)=1$. 
Here, a probability measure $\mu$ on $\Isom(X)$ is said to be non-arithmetic if there exists $n \in \N$ and $g,g' \in \supp(\mu^{\ast n})$ such that $\tau(g) \neq \tau(g')$ where $\tau$ is the translation distance, $\tau(g)=\lim_{n \to \infty}\frac{\kappa(g^n)}{n}$.
\end{remark}

%The proof of the previous lemma will require the following ingredient from the theory of random walks on Gromov-hyperbolic spaces. It is a direct consequence of uniform punctual deviation estimates given in \cite[Proposition 2.12]{BMSS}.

The proof of the previous lemma is based on showing that $f_n(x)-f_n(y)$ converges uniformly to zero (see \eqref{expression}), which imposes the limit to be the average $\sigma_{\mu}^2$ as defined in \eqref{def.variance}. To prove this, we express   $f_n(x)$ as the variance of  $\frac{M_{x,n}}{\sqrt{n}}$ (see \eqref{identity1}). Using Burkholder's inequalities, the proof boils down to showing deviation  inequalities for $\sigma(L_n,x)-\sigma(L_n,y)$ uniformly in $x,y\in \overline{X}^h$ (see \eqref{eq.botim1}). For the latter fact, we will use the following lemma  which is a direct consequence of uniform punctual deviation estimates given in \cite[Proposition 2.12]{BMSS}.

\begin{lemma}[Uniform punctual deviations]\label{lemma.unif.punctual}
Keep the hypotheses of Proposition \ref{prop.bracket.large.dev}. Then there are constants $C, \beta >0$ such that for any $k \in \N$ and any $x \in \overline{X}^h$, $R>0$ we have 
$$
\mathbb{P}( \kappa(L_k)-\sigma(L_k,x) > R) \leq C e^{-\beta R}.
$$
\end{lemma}
\begin{proof}

Notice that for $g\in \Isom(M)$ and  $x \in \overline{X}^h$, by \eqref{gromov.extended.space.boundary} we have $\kappa(g)-\sigma(g,x)=2(g^{-1}o,x)$. In particular, when $x \in X$, the statement  precisely corresponds to \cite[Proposition 2.12]{BMSS} applied with the image $\check{\mu}$ of $\mu$ by the map $g \mapsto g^{-1}$ on $\Isom(X)$. To extend it to $\overline{X}^h$, given $x \in \overline{X}^h$, let $x_n$ be a sequence in $X$ such that $x_n \to x$ in $\overline{X}^h$. In particular, by continuity of $\sigma(g,\cdot)$,  we have $\kappa(g)-\sigma(g,x)=\lim_{n \to \infty}2(g^{-1}o,x_n)$. Therefore, given $R>0$, 
\begin{equation}\label{eq.tofatou}
\mu^{\ast k}\{g : \kappa(g)-\sigma(g,x) >R \}=\check{\mu}^{\ast k} \{g : \lim_{n \to \infty}(go,x_n) >R/2\},
\end{equation}
for every $k \in \N$. Denoting by $h_n(.)$, the map $g \mapsto \mathds{1}_{(go,x_n)>R/2}$, by \eqref{eq.tofatou} we have
$$
\mathbb{P}( \kappa(L_k)-\sigma(L_k,x) > R)=\int \lim_{n \to \infty}h_n(g) d\check{\mu}^{\ast k}(g)=\lim_{n \to \infty} \check{\mu}^{\ast k} \{g: (go,x_n) >R/2\},
$$
where we used dominated convergence in the last equality.
%which, by Fatou lemma, is bounded above by
%$$
%\liminf_{n \to \infty} \int h_n(g) d\check{\mu}^{\ast k}=\liminf_{n \to \infty} \check{\mu}^{\ast k} \{g: (go,x_n) >R/2\}.
%$$
Hence the statement follows from \cite[Proposition 2.12]{BMSS}.
\end{proof}

\begin{proof}[Proof of Lemma \ref{lemma.unif.conv}]
First, we reduce the problem to showing that \begin{equation}\label{uniform.distance}
f_n(x)-f_n(y)\to 0
\end{equation} 
uniformly in $x$ and $y$ in $\overline{X}^h$. Indeed, let us assume for a while that \eqref{uniform.distance} holds.  Fix any $\mu$-stationary measure $\nu$ on $\overline{X}^h$ (the latter exists by compactness of $\overline{X}^h$). We have  for every $n\in \N$, 
\begin{equation}\label{expression}f_n(x)=\frac{1}{n}\sum_{j=1}^n{\E[\phi(Z_j)| Z_0=x]}
= \frac{1}{n}\sum_{j=1}^n{\E[\int{(\sigma_0(g, L_{j} \cdot x)-\ell_{\mu})^2 d\mu(g)}]}.
\end{equation}
Since $\nu$ is $\mu^{\ast n}$-stationary for every $n\in \N$, we deduce that for every $n\in \N$, 
$$
\int_{\overline{X}^h}{f_n(x) d\nu(x)}=
\iint_{G\times \overline{X}^h}{(\sigma_0(g,x) - \ell_{\mu})^2 d\mu(g) d\nu(x)}:=\sigma^2_{\mu,\nu}.
$$
Let $\epsilon>0$ and $y\in \overline{X}$. We can find $n_0$ depending only on $\epsilon$ such that for every $n\geq n_0$ and   for every $x\in \overline{X}^h$, $f_n(x)-\epsilon\leq f_n(y)\leq f_n(x)+\epsilon$. Integrating on both sides with respect $d\nu(x)$, we obtain that $|f_n(y)-\sigma_{\mu,\nu}^2|\leq \epsilon$ for every $n\geq n_0$, concluding the   proof of the uniform convergence of the sequence functions  $(f_n)_{n\in \N}$ towards $\sigma_{\mu,\nu}^2$. It also shows that $\sigma^2_{\mu, \nu}$ is independent of the choice of the stationary measure $\nu$. 

From now on, we focus on showing the convergence \eqref{uniform.distance} uniformly in $x,y\in \overline{X}^h$. Since $(M_{x,n}^2-\langle M_x\rangle_n)_{n\in \N}$ is a martingale starting at zero, we have that  $\E[\langle M_x\rangle_n]=\E[M_{x,n}^2]$ for every $n\in \N$ so that by \eqref{expression}: 
\begin{equation}\label{identity1}
f_n(x)=\frac{1}{n}\E[M_{x,n}^2]=\E\left[\left( \frac{M_{x,n}}{\sqrt{n}}\right)^2\right]
\end{equation}
Let us check that the sequence  $\{(\frac{M_{x,n}}{\sqrt{n}})^2 \, | \, n\in \N, x\in \overline{X}^h\}$ is uniformly bounded in $L^p$ for every $p> 1$; and hence in particular uniformly integrable. Indeed, 
by Burkholder's inequality (\cite[Theorem 9]{burkholder}), we have 
for every $k>2$, 
\begin{equation}\label{eq.apply.burkholder}
\E[|M_{x,n}|^k]\leq C_k \E[[M_x]_n^{\frac{k}{2}}],
\end{equation}
where $C_k>0$ is a constant depending only on $k$, $[M_x]_n=\sum_{j=1}^n{(\Delta_j M_{x})^2}$ is the quadratic variation of $M_x$.
By Jensen's inequality, we have 
$$[M_x]_n^{\frac{k}{2}}\leq n^{\frac{k}{2}-1}\sum_{j=1}^n{|\Delta_j M_{x}|^k},
$$
so that 
\begin{equation}\label{intt1}
\E[[M_x]_n^{\frac{k}{2}}]\leq 
n^{\frac{k}{2}-1}\sum_{j=1}^n{\E[|\Delta_j  M_{x}|^k]}\leq n^{\frac{k}{2}} 
\sup_{j\in \N}{\E[|\Delta_j M_{x}|^k}].
\end{equation}
Remembering that $\sigma(g,x)=\sigma_0(g,x)-\psi(g\cdot x)+\psi(x)$, $|\sigma(g,x)|\leq \kappa(g)$,  and that $\psi$ is bounded on $\overline{X}^h$, we get that for every $n\in \N$, 
\begin{equation}\label{intt2}
|\Delta_n M_{x}|=|\sigma_0(X_n, L_{n-1}\cdot x)-\ell_{\mu}|\leq \kappa(X_n) + \ell_{\mu} +2||\psi||_{\infty},
\end{equation}
Since the $X_i$'s have the same distribution, by plugging \eqref{intt2} and 
\eqref{intt1} in \eqref{eq.apply.burkholder} we get
$$
\E[M_{x,n}^k]\leq C_k n^{\frac{k}{2}}\E[|\kappa(X_1) + \ell_{\mu} +2||\psi||_{\infty}|^k].
$$ 
The right-hand-side is finite (since $\mu$ has a finite moment of any order $k>2$) and does not depend neither on $n$ nor on $x$, showing the boundedness in $L^{k/2}$ of $(\frac{M_{x,n}}{\sqrt{n}})^2$ uniformly in $n$ and $x$.  

Let now $\epsilon>0$. It follows from the uniform integrability of the family $\frac{M_{x,n}}{\sqrt{n}}$ that there exists $L(\epsilon)>0$ such that for every $n\in \N$ and  $x,y\in \overline{X}^h$ we have
$$
\p( \max\{\frac{|M_{x,n}|}{\sqrt{n}}, \frac{|M_{y,n}|}{\sqrt{n}}\}>L(\epsilon))<\epsilon.
$$
Using Lemma \ref{lemma.unif.punctual} together with the fact that $\sigma$ differs from $\sigma_0$ from a bounded function on $\overline{X}^h$, we obtain some $T(\epsilon)>0$ such that for every $n\in \N$,  $x,y\in \overline{X}^h$, 
\begin{equation}\label{eq.botim1}
\p(|\sigma_0(L_n, x)-\sigma_0(L_n,y)|>T(\epsilon))\leq \epsilon.
\end{equation}
Hence, if  $A_{x,y,n}$ denotes  the event 
$$
A_{x,y,n}:= \left\{|\sigma_0(L_n, x)-\sigma_0(L_n, y)|>T(\epsilon) \right\} \cup   \left\{\max\{\frac{|M_{x,n}|}{\sqrt{n}}, \frac{|M_{y,n}|}{\sqrt{n}}\}> L(\epsilon)\right\},
$$ we have for every $n\in \N$, $x,y\in \overline{X}^h$ that 
$\P(A_{x,y,n})<2\epsilon$. Now we write 
$$
|f_n(x)-f_n(y)|\leq   \underset{a_{x,y,n}}{\underbrace{ \E\left[\left|(\frac{M_{x,n}}{\sqrt{n}})^2 - (\frac{M_{y,n}}{\sqrt{n}})^2\right|\mathds{1}_
{A_{x,y,n}}\right]}}  +  \underset{b_{x,y,n}}{\underbrace{ \E\left[\left|(\frac{M_{x,n}}{\sqrt{n}})^2 - (\frac{M_{y,n}}{\sqrt{n}})^2\right|\mathds{1}_{A_{x,y,n}^C}\right]}}.
$$
Let us estimate $a_{x,y,n}$. By Cauchy--Schwarz inequality, we have for every $n\in \N$, 
$$
a_{x,y,n}^2 \leq 2\max_{x,y}{\E[(\frac{M_{x,n}}{\sqrt{n}})^4]} \,\p(A_{x,y,n})\leq C_4\epsilon,
$$
where $C_4>0$ is a constant independent of $n, x$ and $\epsilon$;   guaranteed by the uniform boundedness in $L^4$ shown at the beginning of the proof. Finally, we estimate $b_{x,y,n}$. Since the function $x\mapsto x^2$ is uniformly continuous on $[-L(\epsilon), L(\epsilon)]$, we can find $\delta(\epsilon)>0$ such that $|t^2-t'^2|<\epsilon$ whenever $|t-t'|<\delta(\epsilon)$ and $\max\{|t|,|t'|\}\leq L(\epsilon)$. Let $n_0(\epsilon)\in \N$ be such that $\frac{T(\epsilon)}{\sqrt{n_0(\epsilon)}}<\delta(\epsilon)$. From the definition of the event $A_{x,y,n}^C$, we deduce that  for every $n\geq n_0(\epsilon)$,  $x,y\in \overline{X}^h$, $b_{x,y,n}<\epsilon$. 
Hence for $n\geq n_0(\epsilon)$, $x,y\in \overline{X}^h$
$$
|f_n(x)-f_n(y)|< \sqrt{C_4 \epsilon}+\epsilon,
$$
which finishes the proof of the uniform convergence \eqref{uniform.distance}. \end{proof}

We end this section with the following consequence of Proposition \ref{prop.bracket.large.dev}. 
%Before proving Proposition \ref{prop.bracket.large.dev}, we hasten to indicate the following consequence of it which is the one we will be using in the sequel. 
In the statement below, for every $x\in \overline{X}^h$, we use the transform  $G_n^a$ introduced in \eqref{transform.G} associated to the martingale $M_x$. To ease the notation, we omit the dependence on $x$ in $G_n^a$.

\begin{corollary}\label{corollary.laplace.bracket}
Suppose $\mu$ has finite exponential moment.  Then for every $x\in \overline{X}^h$,
$$\lim_{a\to +\infty}\limsup_{\lambda\to 0^+}\frac{1}{\lambda} \limsup_{n\to +\infty}\frac{1}{n}\log {\E[\exp(-\lambda (G_n^a - n \sigma_{\mu}^2))]}= 0.$$
\end{corollary}

\begin{proof} 
Let $x\in \overline{X}^h$. The result will follow from Cauchy--Schwarz inequality and the following two estimates 
\begin{equation}\label{1bis}
\limsup_{\lambda\to 0^+}\frac{1}{\lambda} \limsup_{n\to +\infty}\frac{1}{n}\log {\E[\exp(-\lambda (\langle M_x\rangle_n-n\sigma_{\mu}^2))]}\leq 0.
\end{equation}
and 
\begin{equation}\label{backtoidd2}
\limsup_{a\to +\infty} \limsup_{\lambda \to 0^+}\frac{1}{\lambda} \limsup_{n\to \infty} \frac{1}{n} \log \E[\exp(\lambda (  \langle M_x \rangle_n - G_n^a))]\leq 0
\end{equation}
% NOTE THAT IN FACT (I) BELOW WORKS WITH FINITE SECOND MOMENT.
\textbf(i) We start by proving \eqref{1bis}. 
To ease the notation, let $Y_n=\langle M_x\rangle_n-n\sigma_{\mu}^2$. Let $\epsilon>0$. By Proposition \ref{prop.bracket.large.dev}, there exists $\alpha(\epsilon)>0$ and $n_0(\epsilon)\in \N$ such that for every $n\geq n_0(\epsilon)$, 
$$\p(Y_n\leq -n\epsilon)\leq \exp(-\alpha(\epsilon) n).$$
Noticing that $Y_n\geq - n\sigma^2_\mu$, we write 
$$
Y_n=Y_n\mathds{1}_{Y_n\geq 0}+ 
Y_n\mathds{1}_{-n\epsilon \leq Y_n\leq 0}
+ Y_n\mathds{1}_{-n\sigma^2 \leq Y_n\leq -n\epsilon}. 
$$
Since $\E[\exp(-\lambda Y_n) \mathds{1}_{-n\epsilon \leq Y_n\leq 0}]\leq \exp(n\lambda \epsilon)$ and, for $n\geq n_0(\epsilon)$,  
$$
\E[\exp(-\lambda Y_n)\mathds{1}_{-n\sigma^2 \leq Y_n\leq -n\epsilon}]\leq \exp(n\lambda \sigma^2)\p(Y_n\leq -n\epsilon)\leq \exp(n (\sigma^2\lambda-\alpha(\epsilon))),
$$
we get that  for every $n\geq n_0(\epsilon)$,
$$
\E[\exp(-\lambda(\langle M\rangle_n - n\sigma^2))]\leq 1+ \exp(n\lambda \epsilon)+\exp( n\sigma^2 \lambda-\alpha(\epsilon)n).
$$
Keeping $\epsilon$ and $\lambda>0$ (small enough) fixed and we let $n\to +\infty$ and deduce that 
$$
\limsup_{n\to +\infty}
\frac{1}{n}\log \E[\exp(-\lambda(\langle M_x\rangle_n - n\sigma^2))]\leq \max\{\lambda \epsilon, \lambda \sigma^2-\alpha(\epsilon)\}.
$$
Since $\alpha(\epsilon)>0$, we get by letting $\lambda\to 0^+$ (while keeping $
\epsilon$ fixed) that 
$$
\limsup_{\lambda\to 0^+}\frac{1}{\lambda}\limsup_{n\to +\infty}
\frac{1}{n}\log \E[\exp(-\lambda(\langle M_x\rangle_n - n\sigma^2))]\leq \epsilon.
$$
Letting $\epsilon\to 0$, we conclude that 
$$
\limsup_{\lambda\to 0^+}\frac{1}{\lambda}\limsup_{n\to +\infty}
\frac{1}{n}\log \E[\exp(-\lambda(\langle M_x\rangle_n - n\sigma^2))]\leq 0.
$$
This shows \eqref{1bis}.\\
 
\textbf{(ii)} Finally, we show \eqref{backtoidd2}. Using the expression \eqref{transform.G} for $G_n^a$, we see that  
$$ \langle M_x \rangle_n - G_n^a = \sum_{i=1}^n{\E[(\Delta_i M_x)^2 \mathds{1}_{|\Delta_i M_x|>a} | \mathcal{F}_{i-1}] + a \sum_{i=1}^n{|\Delta_i M_x|} \mathds{1}_{|\Delta_i M_x|\geq a}}.$$
Observe  that by  the expression of  
our martingale difference \eqref{our.martingale.difference} and by the decomposition \eqref{cocycle.decomposition},
we have a.s.~ for every $i\in \N$, 
\begin{equation}\label{equation.upper.MDS}|\Delta_i M_x|= |\sigma_0(X_i, L_{i-1}\cdot x )-\ell_\mu|\leq 
\kappa(X_i)+\ell_\mu+||\psi||_{\infty}:=\zeta_i.\end{equation}
Since the $\zeta_i$'s have the same distribution,  
$$\E[(\Delta_i M_x)^2 \mathds{1}_{|\Delta_i M_x|>a} | \mathcal{F}_{i-1}]\leq \E(\zeta_1^2 \mathds{1}_{\zeta_1>a}):=h_{\mu}(a).$$
% \begin{eqnarray}\E[(\Delta_i M_x)^2 \mathds{1}_{|\Delta_i M_x|>a} | \mathcal{F}_{i-1}]   &=& \int{\left(\sigma_0(g, L_{i-1}\cdot x)) - \ell_{\mu}\right)^2 \mathds{1}_{|\sigma_o(g, L_{i-1}\cdot x)-\ell_{\mu}|>a} d\mu(g)}\nonumber\\
 %& \leq &   \int{(\kappa(g) + \ell_{\mu}+||\psi||_{\infty})^2  \mathds{1}_{\kappa(g) + \ell_{\mu}+||\psi||_{\infty}>a} d\mu(g)}:=h_{\mu}(a)\nonumber
% \end{eqnarray}
Observe that the constant $h_{\mu}(a)$ is independent of $i$. Since $\mu$ has finite second moment, we deduce that 
\begin{equation}\label{quedire1}\limsup_{a \to \infty}\limsup_{\lambda \to 0^+}\frac{1}{\lambda} \limsup_{n\to \infty} \frac{1}{n} \log \E[e^{\lambda \sum_{i=1}^n \mathbb{E}[(\Delta_i M_x)^2 \mathds{1}_{|\Delta_i M_x|>a} | \mathcal{F}_{i-1}]} ]\leq \limsup_{a\to \infty}{h_{\mu}(a)}=0.\end{equation}
On the other hand, for every $a>0$,  the random variables $a\zeta_i\mathds{1}_{|\zeta_i|>a}$ are i.i.d random variables. Denote by $\zeta_a$ their common  distribution and  by  $\Lambda_{\zeta_a}$ the  Laplace transform of $\zeta_a$. The latter is differentiable at $0$ for every $a>0$ as $\zeta_1$ has finite exponential moment (because $\mu$ has finite exponential moment). 
It follows that 
\begin{eqnarray}
&&\limsup_{a\to +\infty} \limsup_{\lambda \to 0^+}\frac{1}{\lambda} \limsup_{n\to \infty} \frac{1}{n} \log \E[e^{\lambda a \sum_{i=1}^n{|\Delta_i M_x| \mathds{1}_{|\Delta_i M_x|\geq a}}}]\nonumber\\ &\leq &
\limsup_{a\to +\infty} \limsup_{\lambda \to 0^+}\frac{1}{\lambda} \limsup_{n\to \infty} \frac{1}{n} \log \E[e^{\lambda  \sum_{i=1}^n{ a |\zeta_i| \mathds{1}_{|\zeta_i|\geq a}}}] = 
\limsup_{a\to +\infty} \Lambda_{\zeta_a}'(0).\label{link}\end{eqnarray} But $\Lambda_{\zeta_a}'(0)=a\E[|\zeta_1| \mathds{1}_{|\zeta_1|>a}]\leq \E[\zeta_1^2 \mathds{1}_{|\zeta_1|>a}]\underset{a\to \infty}{\to} 0$. This concludes the proof of \eqref{backtoidd2} and hence the corollary. 
\end{proof}

\section{Proof of the main result}\label{sec.proof}
Having established the submartingale transform from \S \ref{sec.martingale} and the exponential decay of large deviation probabilities of the predictable quadratic variation from \S \ref{sec.large.deviation}, we are now ready to give the proofs of Theorem \ref{thm.main} and Corollary \ref{corol.rate.function}. In fact, we will prove a slightly more general version given by Theorem \ref{thm.main.tech} below.

\subsection{Statement of the main result}
To state the more general version of Theorem \ref{thm.main}, we recall and introduce some notation. We are given a geodesic Gromov-hyperbolic space $X$ with a fixed based point $o \in X$. The Busemann cocycle $\sigma: \Isom(X) \times \overline{X}^h \to \R$ with respect to the base point $o$ is as defined in \S \ref{subsec.hyperbolic}. Given a countably supported probability measure $\mu$ on $\Isom(X)$ and $x \in \overline{X}^h$, we define the upper $\Lambda_x^+$ and lower $\Lambda_x^-$ limit Laplace transforms as
$$
\Lambda_x^+(\lambda):=\limsup_{n \to \infty} \frac{1}{n} \log \mathbb{E}[e^{\lambda\sigma(L_n,x)}],
$$
and
$$
\Lambda_x^-(\lambda):=\liminf_{n \to \infty} \frac{1}{n} \log \mathbb{E}[e^{\lambda\sigma(L_n,x)}].
$$

Whenever $\mu$ has finite exponential moment both functions have values in $\R$ in a neighborhood of $0 \in \R$.

We will omit sub/super-scripts when $x\in X$, indeed, for every $x \in X$, we have $\Lambda_x^+ \equiv \Lambda_x^-\equiv \Lambda_o$ (since $|\sigma(g,x)-\sigma(g,y)|\leq 2 d(x,y)$ for $g\in G$ and $x,y\in X$). This common function $\Lambda=\Lambda_o$ is the notation used in Theorem \ref{thm.main} where we work with the basepoint $x=o \in X$.

\begin{theorem}\label{thm.main.tech}
Let $(X,d)$ be a separable geodesic Gromov-hyperbolic space and $\mu$ a non-elementary probability measure on $\Isom(X)$. Suppose that $\mu$ has a finite exponential moment.  Then for every $x\in \overline{X}^h$, 
$$\lim_{\lambda \to 0} \frac{\Lambda_x^-(\lambda)-\ell \lambda }{\lambda^2}=
\lim_{\lambda \to 0} \frac{\Lambda_x^+(\lambda)-\ell \lambda }{\lambda^2}=\frac{\sigma^2_{\mu}}{2}.$$  
\end{theorem}

Note that this result is a more general version of Theorem \ref{thm.main} from introduction except for the separability assumption on $X$. However, Theorem \ref{thm.main} follows from it since, thanks to \cite[Remark 4]{gruber-sisto-tessera}, we can replace $X$ in Theorem \ref{thm.main} with a separable geodesic subset $X'$ invariant under the action of the group generated by the support of $\mu$ and simply apply Theorem \ref{thm.main.tech} with $X'$ for some $x \in X'$. This then implies Theorem \ref{thm.main} without separability assumption as claimed.

\bigskip

Subsections \S \ref{subsec.lower.bound} and \S \ref{subsec.upper.bound} are devoted to the proof of Theorem \ref{thm.main.tech}. %For every $x \in \overline{X}^h$, in \S \ref{subsec.lower.bound}, we establish the inequality $\lim_{\lambda \to 0} \frac{\Lambda_x^-(\lambda)-\ell \lambda }{2\lambda^2} \geq \sigma^2_{\mu}$ and in \S \ref{subsec.upper.bound}, we prove the inequality $\lim_{\lambda \to 0} \frac{\Lambda_x^+(\lambda)-\ell \lambda }{2\lambda^2} \leq \sigma^2_{\mu}$. Together, they complete the proof of Theorem \ref{thm.main.tech}.
%Finally, in  \S \ref{subsec.corol}, we prove Corollary \ref{corol.rate.function} from introduction.

\subsection{Proof of the lower bound}\label{subsec.lower.bound}
 
Here we prove the following.
\begin{proposition}\label{prop.lower.bound}
Keep the setting of Theorem \ref{thm.main.tech}. Then, for every $x\in \overline{X}^h$ $$\lim_{\lambda \to 0} \frac{\Lambda_x^-(\lambda)-\ell_\mu \lambda }{\lambda^2} \geq \frac{\sigma^2_{\mu}}{2} .$$ 
\end{proposition}
 
\begin{proof}[Proof of Proposition \ref{prop.lower.bound}]

Given a probability measure $\mu$ as in the statement and $x \in \overline{X}^h$, let $M_{x}$ be the martingale given by Lemma \ref{lemma.martingale.dec}. It satisfies
$$
\sigma(L_n,x)-n\ell_\mu=M_{x,n}+O_{x,n}(1),
$$
for every $n \in \N$, where $O_{x,n}(1)$ is a random variable that is bounded (in absolute value) uniformly in $x \in \overline{X}^h$ and $n\in \N$. Let  $\sigma_\mu^2>0$ be as defined in \eqref{def.variance}. Let $x \in \overline{X}^h$ be fixed for the rest of the proof. For every  $\lambda \in \R$, we have
\begin{eqnarray*}\label{observe}\Lambda_x^-(\lambda)-\lambda \ell_{\mu}  & = & \liminf_{n \to \infty}\frac{1}{n}\log \E[e^{\lambda (\sigma(L_n,x) - n\ell_{\mu})}]\nonumber\\
& = &\liminf_{n \to \infty} \frac{1}{n}\log \E[e^{\lambda( M_{x,n} +O_{x,n}(1))}]\\
& = &\liminf_{n \to \infty} \frac{1}{n}\log \E[e^{\lambda M_{x,n}}], \end{eqnarray*}
where we used the fact that the random variables $O_{x,n}(1)$ are bounded below and above uniformly in $n \in \N$. Notice that since $\mu$ has finite exponential moment, for every $\lambda$ in a neighborhood of $0 \in \R$ (independent of $x \in \overline{X}^h$), the last quantity in the above displayed equation is finite.
%By the same token, we also have
%$$
%\Lambda_x^-(\lambda)-\lambda \ell_{\mu}=\liminf_{n \to \infty} \frac{1}{n}\log \E[e^{\lambda M_{x,n}}].
%$$
%We will proceed by a case-by-case analysis to show the equalities in Theorem \ref{thm.main.tech}.

%\bigskip

%\textbf{(i)} 
We first prove that 
\begin{equation}\label{eq.main0}
\frac{\sigma_{\mu}^2}{2}\leq \liminf_{\lambda\to 0^+}{\frac{\Lambda_x^-(\lambda)-\lambda \ell_{\mu}}{\lambda^2}}.
\end{equation}

Let $n\in \N$, $a>0$, and $\lambda>0$ small enough. By Proposition \ref{prop.submart.trans}, we have  
$$
1\leq \E[\exp(\lambda M_{x,n} - \frac{\mathfrak{f}(a\lambda)}{a^2} G_n^a)],
$$
where, we recall, $G_n^a=\sum_{i=1}^n \mathbb{E}[(\Delta_i M_x)^2 1_{|\Delta_i M_x| \leq a}| \mathcal{F}_{i-1}]-a\sum_{i=1}^n |\Delta_iM_x| 1_{|\Delta_i M_x| \geq a}$.
Let $p> 1$. By H\"{o}lder inequality, we get 
$$
1\leq 
\E[\exp(p\lambda M_{x,n})]^{1/p} 
\E[\exp(- q\frac{\mathfrak{f}(a\lambda)}{a^2}G_n^a)]^{1/q},
$$
where $q \geq 1$ satisfies $1/p+1/q=1$.   

Taking logarithm and dividing by $n\lambda^2$, adding and subtracting the term $\frac{\mathfrak{f}(a\lambda)}{(a\lambda)^2}\sigma_{\mu}^2$ gives 
\begin{equation}\label{eq.main1}
0\leq \frac{1}{n p \lambda^2}\log \E[\exp( p\lambda M_{x,n})]+ \frac{1}{nq \lambda^2}\log \E[\exp(- q\frac{\mathfrak{f}(a\lambda)}{a^2}  (G_n^a-n\sigma_{\mu}^2))]-\frac{\mathfrak{f}(a\lambda)}{(a\lambda)^2}\sigma_{\mu}^2.
\end{equation}

Using the elementary fact $\liminf_{n\to +\infty}(a_n+b_n)\leq \liminf_n a_n + \limsup_n b_n$  for real sequences $a_n$ and $b_n$, letting $n \to \infty$ in \eqref{eq.main1}, we get 
$$
0\leq  p \frac{\Lambda_x^-(\lambda p)- \lambda p\ell_{\mu}}{(\lambda p)^2} +  \frac{1}{ q\lambda^2}\limsup_{n\to +\infty} \frac{1}{n}\log \E[\exp(- q\frac{\mathfrak{f}(a\lambda)}{a^2}  (G_n^a-n\sigma_{\mu}^2))]- \frac{\mathfrak{f}(a\lambda)}{(a\lambda)^2}\sigma_{\mu}^2.
$$
Letting $\lambda\to 0^+$ while noting that  $\mathfrak{f}(a\lambda) \underset{\lambda\to 0}{\sim} (a\lambda)^2/2$ and in particular $\eta=q\frac{\mathfrak{f}(a\lambda)}{a^2} \underset{\lambda \to 0}{\to} 0^+$, we obtain:  
$$
\frac{\sigma_{\mu}^2}{2}\leq p \liminf_{\lambda\to 0^+}{\frac{\Lambda_x^-(\lambda)-\lambda \ell_{\mu}}{\lambda^2}}
+ \frac{1}{2}\limsup_{\eta\to 0^+}\frac{1}{\eta}
\limsup_{n\to +\infty}\frac{1}{n}\log\E[\exp(-\eta (G_n^a-n\sigma_{\mu}^2))].
$$
Letting $a \to \infty$, we deduce from Corollary \ref{corollary.laplace.bracket} that we have
$$
\frac{\sigma_\mu^2}{2} \leq p  \liminf_{\lambda\to 0^+}{\frac{\Lambda_x^-(\lambda)-\lambda \ell_{\mu}}{\lambda^2}}
$$
The desired inequality \eqref{eq.main0} is now proved by taking $p \to 1$.

The inequality 
\begin{equation}\label{eq.main}
\frac{\sigma_{\mu}^2}{2}\leq \liminf_{\lambda\to 0^-}{\frac{\Lambda_x^-(\lambda)-\lambda \ell_{\mu}}{\lambda^2}} %\qquad \text{and} \qquad \limsup_{\lambda \to 0^-}\frac{\Lambda^+_x(\lambda)-\lambda \ell_{\mu}}{\lambda^2}\leq \frac{\sigma_{\mu}^2}{2}
\end{equation}
is proven in precisely the same way replacing the martingale $M_{x}$ by the martingale $-M_{x}$ using the fact that both martingales have same transforms $G^a$. This completes the proof of Proposition \ref{prop.lower.bound}.
\end{proof}

\subsection{Proof of the upper bound} \label{subsec.upper.bound}
Here we prove the following.
\begin{proposition}\label{prop.upper.bound}
Keep the setting of Theorem \ref{thm.main.tech}. Then, for every $x \in \overline{X}^h$ \begin{equation}\label{eq.upper.bound}\lim_{\lambda \to 0} \frac{\Lambda_x^+(\lambda)-\ell_\mu \lambda }{\lambda^2} \leq \frac{\sigma^2_{\mu}}{2}.\end{equation}
\end{proposition}
The proof is based on showing that for large $n \in \N$ the random variable $\frac{1}{n}\sigma(L_n,x) -  \ell_\mu$ has a subgaussian behaviour in a neighborhood of $0$. This is shown in the following proposition which controls the limit Laplace transform of the sequence of random variables $ \frac{1}{n}\sigma(L_n,x) - \ell_\mu$. The proof is based on the martingale decomposition given in Lemma \ref{lemma.martingale.dec} and standard techniques for concentration results for martingales (in particular \cite[Theorem 2.19]{wainwright}). With the notation of \S \ref{sec.hyperbolic}, the main tool for the proof of Proposition \ref{prop.lower.bound} is the following. 
\begin{proposition}\label{proposition.laplace.control}
Let \begin{equation}\label{eq.vmu}
v(\mu):=\sup_{\xi\in \overline{X}^h}{\E\left[\left(\sigma_0(X_1,\xi)-\ell_\mu\right)^2 \right]}=\sup_{\xi\in \overline{X}^h}{\E[M_{\xi,1}^2]}.
\end{equation}
Then there exists  $C>0$ such that for every $\epsilon>0$, there exists $b>0$ such that for every $|\lambda|<\frac{v(\mu)}{b}$, every $n\in \N$ and every $x\in \overline{X}^h$, 
\begin{equation}\label{eq.laplace.control}\E\left[e^{\lambda (\sigma(L_n,x) - n \ell_\mu)}\right] \leq \exp\left(\frac{\lambda^2 (v(\mu)+\epsilon) n}{2}+ C |\lambda|\right).\end{equation}
\end{proposition}
%This proposition recovers Proposition 2.1 in Aoun--Sert PTRF in an easier way. The detour made in A-S is to pass by Prop 2.3 which does not (cannot due to its generality) exploit the random walk structure behind the martingales. Here we directly exploit it and prove Prop 2.1. Note that in A-S in Prop 2.3 we could have used other martingale large deviation inequalities (than Hoeffding, like de la Pena, Watleb-Liu), but then checking their hypotheses becomes an issue to deal with.

This proposition will yield \eqref{eq.upper.bound} with $\sigma_\mu^2$ replaced by the larger quantity $v(\mu)$. To obtain \eqref{eq.upper.bound}, we will use an acceleration technique speeding up the random walk, see the proof of Proposition \ref{prop.upper.bound}. 

We now proceed with proving Proposition \ref{proposition.laplace.control}. The proof is based on the following control of the conditional expectation of the martingale difference $\Delta M_x$: 

\begin{lemma}\label{lemma.martingale.condition}
For every $\epsilon>0$, 
there exists a constant $b>0$ such that for every $|\lambda|<b$, $n\in \N$ and $x\in \overline{X}^h$, the following inequality holds almost surely: 
\begin{equation*}
\E\left[\exp(\lambda \Delta_n M_x)\,|\,\mathcal{F}_{n-1} \right]
\leq \exp\left(\frac{\lambda^2 (v(\mu)+\epsilon)}{2}\right).
\end{equation*}
\end{lemma}
% Ici, on detourne "la difficulte que de la Pena" car notre martingale est speciale: 
%A savoir, en general, d'un martingale on peut pas esperer une borne superieure deterministe (comme dans ce lemme). Ici, on exploite la structure iid qui est derriere random walks pour obtenir this strong conditional expectation inequality with a deterministic upper bound.

\begin{proof}
By expanding the expression \eqref{equation.martingale} of the martingale $M_x$ and taking conditional expectation, it suffices to show that for every $\epsilon>0$, there exists a constant $b>0$ such that for every $\xi \in \overline{X}^h$ and $|\lambda|<b$
\begin{equation}\label{eq.b.lambda.epsilon}
\int e^{\lambda(\sigma_0(g,\xi)-\ell_\mu)} d\mu(g) \leq \exp\left(\frac{\lambda^2 (v(\mu)+\epsilon)}{2}\right).
\end{equation}
Using the exponential moment assumption on $\mu$, let $\alpha>0$ be such that $\int e^{\alpha \kappa(g) d\mu(g)}<\infty$. Thanks to \eqref{equation.martingale}, we have that for every $g \in \Isom(X)$ and $\xi \in \overline{X}^h$,  $|\sigma_0(g,x)-\ell_\mu| \leq \kappa(g) + 2(\|\psi\|+\ell_\mu)$. Therefore, for every $|\lambda|<\alpha$, using dominated convergence, we have
$$
\int e^{\lambda(\sigma_0(g,\xi)-\ell_\mu)} d\mu(g)=1+\frac{\lambda^2}{2}\E[(\sigma_0(X_1,\xi)-\ell_\mu)^2]+\sum_{k=3}\frac{\lambda^k}{k!}\E[(\sigma_0(X_1,\xi)-\ell_\mu)^k],
$$
where we have used the fact that $\sigma_0(X_1,\xi)-\ell_\mu=M_{\xi,1}$ has mean zero (as the cocycle $\sigma_0$ has constant drift). Now using again the fact that $|\sigma_0(g,x)-\ell_\mu| \leq \kappa(g) + 2(\|\psi\|+\ell_\mu)$ and that $\mu$ has finite exponential moment, we get that there exists  $C>0$ (independent of $\xi$) such that for every $|\lambda|<\alpha$ such that 
$$
\int e^{\lambda(\sigma_0(g,\xi)-\ell_\mu)} d\mu(g)\leq 1+\frac{\lambda^2 v(\mu)}{2}+ C \lambda^3.
$$
This readily implies \eqref{eq.b.lambda.epsilon} and hence finishes the proof.
\end{proof}

\begin{proof}[Proof of Proposition \ref{proposition.laplace.control}]
By the tower property of the conditional expectation, we deduce from Lemma  \ref{lemma.martingale.condition} that for every $\epsilon>0$,    $|\lambda|<b$ (where $b=b(\mu,\epsilon)$ is given by the aforementioned lemma),   every $n\in \N$ and $x\in \overline{X}^h$, 
$$\E[ e^{\lambda M_{x,n}}] 
= \E[e^{\lambda M_{x,n-1}} \E[e^{\lambda \Delta M_{x,n}} |\mathcal{F}_{n-1}]]
\leq \exp\left(\frac{\lambda^2(v(\mu)+\epsilon)}{2}\right) \E[e^{\lambda M_{x,n-1}}].$$
Iterating the same process, we deduce that 

$$\E[\exp(\lambda M_{x,n})]\leq \exp\left(\frac{ \lambda^2 (v(\mu)+\epsilon) n}{2}\right).$$
Finally, recall that $\sigma(L_n,x)-n\ell_\mu=M_{x,n}+R_{x,n}$ where $|R_{x,n}|\leq 2||\psi||_{\infty}:=C$. This finishes the proof of the proposition. 
\end{proof}

A remark on the proof Proposition \ref{proposition.laplace.control} is in order.

\begin{remark}
Given a martingale $M$ with unbounded differences, controlling various quantities involving the \emph{conditional expectation} of the martingale difference sequence $\Delta M$ is generally an important step to prove concentration results for the martingale $M$; see the works of de La Pe\~{n}a \cite{de.la.pena}, Dzhaparidze--van Zanten \cite{DZ}, Fan--Grama--Liu \cite{fan.grama.liu} and Liu--Watbled \cite{liu-watbled} who prove Bennett--Bernstein type concentration inequalities generalizing results of Freedman \cite{Freedman} to the case of unbounded differences. Proposition \ref{proposition.laplace.control} avoids using these more sophisticated results thanks to Lemma \ref{lemma.martingale.condition} which, exploiting the special form of our martingales (namely, coming from an iid random walk on a group), gives a deterministic bound for the exponential of the conditional expectation.
\end{remark}

We are now ready to give
\begin{proof}[Proof of Proposition \ref{prop.upper.bound}]
Using Proposition \ref{proposition.laplace.control}
and taking logarithm and dividing by $n$ on both sides of \eqref{eq.laplace.control}, letting first $n\to +\infty$, then $\lambda\to 0$, and finally $\epsilon\to 0$, we get that 
\begin{equation}\label{eq.upper.bound2} \limsup_{\lambda\to 0}
\frac{\Lambda_x^+(\lambda)-\lambda \ell_\mu}{\lambda^2}\leq \frac{v(\mu)}{2}. 
\end{equation}
This yields \eqref{eq.upper.bound} with $\sigma_{\mu}^2$ replaced with the larger quantity $v(\mu)$. We now employ an acceleration trick. More precisely, consider, for every $k\in \N$, the probability measure $\mu^{\ast k}$ (distribution of $L_k$), which is a non-elementary probability measure with finite exponential moment. Denote by $\Lambda(\mu^{\ast k},.)$ the Laplace transform based at $x=o$ for the $\mu^{\ast k}$-random walk $(L_{nk})_{n \in \N}$. In particular, $\Lambda(\mu,.)=\Lambda(.)$. Applying \eqref{eq.upper.bound2} for the
$\mu^{\ast k}$-random walk, we deduce that  for every $k\geq 1$, 
$$\limsup_{\lambda\to 0}\frac{\Lambda({\mu^{\ast k}},\lambda)- \lambda \ell_{\mu^{\ast k}}}{\lambda^2}\leq \frac{ v(\mu^{\ast k})}{2}
$$
%and 
%$$\liminf_{\lambda\to 0}\frac{I_{\mu^{\ast k}}(\ell(\mu^{\ast k})+\lambda)}{\lambda^2}\geq \frac{1}{2 v(\mu^{\ast k})}.
%$$
Here, in straightforward way, we have $\ell_{\mu^{\ast k}}=k \ell_\mu$ %a straightforward calculation using the definition \eqref{eq.defn.laplace} of the limit Laplace transform, % ---> There is no calculation to do
%we have 
and $\Lambda({\mu^{\ast k}},.)=k \Lambda(.)$. Hence, for every $k \geq 1$, 
$$
\limsup_{\lambda\to 0}\frac{\Lambda(\lambda)- \ell_\mu \lambda}{\lambda^2}\leq \frac{ v(\mu^{\ast k})}{2k}.
$$
It remains to check that 
\begin{equation}\label{equation.variance.limit}
\lim_{k\to \infty}{\frac{v(\mu^{\ast k})}{k}}=\sigma_{\mu}^2.\end{equation}
By definition of $v(\mu^{\ast k})$ given in \eqref{eq.vmu} and using \eqref{identity1} (with the notation of Lemma \ref{lemma.unif.conv}), we get that for every $k \geq 1$, 
$$\frac{v(\mu^{\ast k})}{k}=\sup_{\xi\in \overline{X}^h}f_k(\xi).$$ Finally, the uniform convergence given in Lemma \ref{lemma.unif.conv} for the sequence $(f_k)_{k\in \N}$ implies  \eqref{equation.variance.limit} and finishes the proof of the proposition.
\end{proof}

\subsection{Proof of Corollary \ref{corol.rate.function}}\label{subsec.corol}
If $\sigma_\mu=0$, it is easy to deduce from Remark \ref{rk.sigma.positive} that the rate function $I$ satisfies $I(\ell_\mu)=0$ and $I(x)=\infty$ for every $x \in \R \setminus \{\ell_\mu\}$ and hence Corollary \ref{corol.rate.function} is true in that case. We therefore suppose $\sigma^2_\mu>0$. To treat this case, we will use some standard terminology from convex analysis, for which we refer the reader to \cite{convexity.book}. Let, as usual, $\Lambda$ denote the limit Laplace transform of the sequence $\frac{1}{n} d(L_n \cdot o,o)$. Note that $\Lambda$ is convex (as it follows by a direct application of H\"{o}lder inequality), and, thanks to the finite exponential moment assumption, it takes finite values on an interval of type $(-\infty,\alpha)$ with $\alpha>0$ and hence it is continuous on this interval. Let $\Lambda^\ast$ be its Fenchel--Legendre transform. By Theorem \ref{thm.main} and \cite[Proposition 6.1.2]{convexity.book}, we have $\partial \Lambda^\ast(\ell_\mu)=\{0\}$ where $\partial \Lambda^\ast$ is the multi-valued subdifferential function of $\Lambda^{\ast}$. Moreover, by Theorem \ref{thm.main}, $\Lambda$ has a second-order development at $0$ and therefore, by \cite[Theorem 5.1.2]{convexity.book} its subdifferential $\partial \Lambda$ is differentiable in the sense of \cite[Definition 5.1.1]{convexity.book}. Since also $\sigma_\mu^2>0$, we can apply \cite[Proposition 6.2.5]{convexity.book} (see also \cite[Proposition 4.5]{gianluca}) and deduce that $\Lambda^\ast$ satisfies 
\begin{equation}\label{eq.expand.lambdastar}
\Lambda^{\ast}(\ell_\mu+ x)=\frac{1}{2\sigma_\mu^2}x + o(x^2)
\end{equation}
as $x \to 0$.
Now, let $\alpha>0$ be the constant appearing in the finite exponential moment condition, i.e.~ $\int e^{\alpha d(g \cdot o,o)} d\mu(g)<\infty$.  
Then, it follows by Varadhan's integral Lemma \cite[Theorem 4.3.1]{dembo-zeitouni} for every $\lambda<\alpha$, we have  $I^\ast(\lambda)=\Lambda(\lambda)$, where $I^\ast$ is the Fenchel--Legendre transform of $I$ and $\Lambda$. But since the second-order term $\sigma^2_\mu/2$ in the second-order expansion of $\Lambda$ at $0$ (given by Theorem \ref{thm.main}) is positive, it follows  that the Fenchel--Legendre transforms of $I^\ast$ and $\Lambda$ coincides in a neighborhood of $\ell_\mu$, i.e.
\begin{equation}\label{eq.Istarstar}
I^{\ast \ast}(x)=\Lambda^{\ast}(x)  
\end{equation}
for every $x \in (\ell_\mu-\beta,\ell_\mu+\beta)$ for some $\beta>0$. But since by \cite[Theorem 1.1]{BMSS}, the function $I$ is convex (and lower semi-continuous), thanks to Fenchel--Legendre duality, we have $I \equiv I^{\ast \ast}$ and hence the corollary follows from \eqref{eq.expand.lambdastar} and \eqref{eq.Istarstar}. \qed

\begin{remark}[On finite time large deviation estimates]
Corollary \ref{corol.rate.function} is an asymptotic statement obviously in its expression (as $\lambda \to 0$) but also concerning the rate function $I$ itself (which controls, from below and above, the exponential rate of decay of probabilities of large deviations of $\frac{1}{n}d(L_n \cdot o,o)-\ell_\mu$ as $n \to \infty$). In regard to giving upper bounds for the large deviation probabilities, Corollary \ref{corol.rate.function} parallels Proposition \ref{prop.upper.bound}. However, in the spirit of concentration estimates, as in the proof of Proposition \ref{prop.upper.bound}, we could have directly used Proposition \ref{proposition.laplace.control} together with the Chernoff bound, to obtain \textit{finite time estimates} for the large deviations of $\frac{1}{n}d(L_n \cdot o,o)-\ell_\mu$\footnote{This finite time estimates then can be used, with the acceleration trick, to prove $\lim_{\lambda \to 0}\frac{I(\ell_\mu+\lambda)}{\lambda^2} \geq \frac{1}{2\sigma_\mu^2}$}. This is in line with the recent work \cite{aoun-sert} where, under additional assumptions, the appearing constants are made explicit (e.g.~ relating with the spectral radius of the probability measure $\mu$ in the regular representation $L^2(G)$ of the isometry group $G$).
\end{remark}

\subsection{Concluding remarks and questions}\label{subsec.questions}
In this final part, we include two questions motivated by our results and and make some brief comments on them.
\subsubsection{Limit Laplace transform of the Busemann cocycle}
As a direct consequence of Theorem \ref{thm.main.tech}, we have that the functions $\Lambda_x^+$ and $\Lambda_x^-$ have the same derivatives at $0$ for every $x\in \overline{X}^h$. Moreover, it is not hard to see that $\Lambda_x^+=\Lambda_y^-$ on $[0,+\infty)$ for every $x,y \in \overline{X}^h$. These suggest the following questions:

\begin{question}\label{question}
Is it true that  $\Lambda_x^+=\Lambda_x^-$ for every $x\in \overline{X}^h$?  More importantly, does there exist a neighborhood of $0$ such that $\Lambda_x^+=\Lambda_y^+$ for every $x,y\in \overline{X}^h$ (and similarly $\Lambda_x^-=\Lambda_y^-$)?
\end{question}

The answer to Question \ref{question} is positive for $x,y \in \partial^h X$ in standard cases when an analytic approach can be implemented. These include random walks on free groups or on classical hyperbolic spaces $\mathbb{H}^n$. Regarding the last part of the question, we note that there are simple examples which show that one cannot ask that the functions $\Lambda_x$ and $\Lambda_y$ coincide throughout the region where they are finite/well-defined -- take for example the  random walk on the group $\mathrm{F}_2=\langle a, b\rangle$ driven by the measure $\mu=\frac{1}{2}(\delta_a+\delta_b)$ and consider $x=a^{+\infty} \in \partial \mathrm{F}_2$ and $y=a^{-\infty}\in \partial \mathrm{F}_2$.\\

\subsubsection{Second-order expansion below the drift without exponential moment}
The rate function $I$ appearing in \eqref{eq.ldp} for $\frac{1}{n}\kappa(L_n)$ exists without any moment assumption \cite[Theorem 2.8]{BMSS}. Moreover, if $\mu$ fails to have finite exponential moment, then the rate function $I$ vanishes on $[\ell_\mu,+\infty)$ (see \cite[Remark 3.2]{BMSS}). On the other hand, it follows from Gou\"{e}zel's \cite[Theorem 1.2]{gouezel.first.moment} that $I$ is positive on $[0,\ell_\mu)$ when $\mu$ has finite first moment. This suggests the following  question 

\begin{question}
Suppose $\mu$ is a non-elementary probability measure with finite second order moment. Is it true that 
$$\lim_{\lambda\to 0^-}\frac{I(\lambda+\ell_{\mu}) }{\lambda^2}=\frac{1}{2\sigma_\mu^2} ?$$
\end{question}
Moreover, we note that thanks to Benoist--Quint \cite[\S 5]{BQ.hyperbolic}, the definition of the variance $\sigma_\mu^2$ given in \eqref{def.variance} even makes sense under the finite first moment hypothesis supposing that the isometry group $\Isom(X)$ acts cocompactly on $X$. Therefore, this suggests the subsequent question as to whether the second-order term in the second-order expansion of $I$ below the drift vanishes when $\sigma_\mu^2=\infty$. Similar questions can be asked about the second-order expansion of the limit Laplace transform $\Lambda$ below zero. 
%Note that this question makes sense whenever $\sigma_{\mu}^2 $ as defined in \eqref{def.variance} is well--defined, which could hold under weaker assumptions than   finite second order moment (see \cite[\S 5]{BQ.hyperbolic}). 

%\bibliographystyle{abbrv} 
%\bibliography{biblio.bib}
\end{document}